\documentclass[a4paper,leqno,11pt,twoside]{amsart}
\usepackage{amsmath,amsfonts,amssymb,amsthm,amscd}
\usepackage[utf8]{inputenc}
\usepackage[english]{babel}
\usepackage[colorlinks=true,citecolor=blue, urlcolor=blue, linkcolor=blue,pagebackref]{hyperref}
\usepackage[top=1.25in,bottom=1.25in,left=1.in,right=1.in]{geometry} 
\usepackage{graphicx}
\usepackage{wrapfig}
\usepackage{paralist}
\usepackage{tabto}
\usepackage{standalone}
\usepackage{xfrac}
\usepackage{float}
\usepackage{tikz-cd}
\usepackage{mathtools}

\usepackage{pgf,tikz}

\usetikzlibrary{arrows,chains,
 decorations.pathreplacing,
shapes,%
}

\usepackage{enumerate}
\usepackage{xcolor}
\usepackage{aliascnt}

\usepackage[normalem]{ulem}

\theoremstyle{plain}
\newtheorem{thm}{Theorem}[section]

\newtheorem{thmIntr}{Theorem}

\newaliascnt{propIntr}{thmIntr}

\newaliascnt{corIntr}{thmIntr}
\newtheorem{corIntr}[corIntr]{Corollary}

\newaliascnt{QU}{thm}
\newtheorem{QU}[QU]{Question}
\aliascntresetthe{QU}

\newaliascnt{conj}{thm}
\newtheorem{conj}[conj]{Conjecture}
\aliascntresetthe{conj}

\newaliascnt{lem}{thm}
\newtheorem{lem}[lem]{Lemma}
\aliascntresetthe{lem}

\newaliascnt{cor}{thm}
\newtheorem{cor}[cor]{Corollary}
\aliascntresetthe{cor}

\newaliascnt{prop}{thm}
\newtheorem{prop}[prop]{Proposition}
\aliascntresetthe{prop}

\theoremstyle{definition}

\newaliascnt{rem}{thm}
\newtheorem{rem}[rem]{Remark}
\aliascntresetthe{rem}

\newaliascnt{defn}{thm}

\aliascntresetthe{defn}

\newaliascnt{ex}{thm}

\aliascntresetthe{ex}

\addto\extrasenglish{

}

\numberwithin{equation}{section}

\def\bP{\ensuremath{\mathbb{P}}}
\def\bQ{\ensuremath{\mathbb{Q}}}

\def\bZ{\ensuremath{\mathbb{Z}}}
\def\bC{\ensuremath{\mathbb{C}}}

\def\cE{\ensuremath{\mathcal{E}}}
\def\cF{\ensuremath{\mathcal{F}}}
\def\cG{\ensuremath{\mathcal{G}}}
\def\cH{\ensuremath{\mathcal{H}}}
\def\cI{\ensuremath{\mathcal{I}}}
\def\cK{\ensuremath{\mathcal{K}}}

\def\cM{\ensuremath{\mathcal{M}}}
\def\cN{\ensuremath{\mathcal{N}}}
\def\cO{\ensuremath{\mathcal{O}}}
\def\cP{\ensuremath{\mathcal{P}}}
\def\cQ{\ensuremath{\mathcal{Q}}}

\def\cS{\ensuremath{\mathcal{S}}}
\def\cT{\ensuremath{\mathcal{T}}}
\def\cU{\ensuremath{\mathcal{U}}}

\def\ev{\ensuremath{\mathrm{ev}}}

\def\hat{\widehat}
\def\hL{{\widehat{L}}}
\def\tS{{\widetilde{S}}}

\def\div{\mathrm{div}}

\def\Db{\mathop{\mathrm{D}^{\mathrm{b}}}\nolimits}

\DeclareMathOperator{\Pic}{Pic}
\DeclareMathOperator{\Supp}{supp}

\DeclareMathOperator{\ext}{ext}

\DeclareMathOperator{\rk}{rk}

\DeclareMathOperator{\im}{Im}
\DeclareMathOperator{\ch}{ch}

\DeclareMathOperator{\coker}{coker}

\DeclareMathOperator{\Gr}{Gr}
\DeclareMathOperator{\Fl}{Fl}
\DeclareMathOperator{\OGr}{OGr}

\DeclareMathOperator{\SO}{SO}
\DeclareMathOperator{\GL}{GL}

\DeclareMathOperator{\Sym}{Sym}
\DeclareMathOperator{\pr}{pr}

\DeclareMathOperator{\Sing}{Sing}

\def\loccit{{\it loc. cit.}}

\def\setminus{\smallsetminus}
\def\emptyset{\varnothing}

\def\longarrow#1#2{\mathchoice{#2}{#1}{#1}{#1}}
\def\to{\longarrow{\rightarrow}{\longrightarrow}}
\def\into{\longarrow{\hookrightarrow}{\lhook\joinrel\longrightarrow}}
\def\onto{\longarrow{\twoheadrightarrow}{\relbar\joinrel\twoheadrightarrow}}
\let\shortmapsto\mapsto
\def\mapsto{\longarrow{\shortmapsto}{\longmapsto}}
\def\bw#1{{\mathchoice%
 {\textstyle{\bigwedge\mkern-4.5mu^{#1}\mkern1mu}}%
 {\textstyle{\bigwedge\mkern-4.5mu^{#1}\mkern1mu}}%
 {\scriptstyle{\bigwedge\mkern-5mu^{#1}}}%
 {\scriptscriptstyle{\bigwedge\mkern-5mu^{#1}}}%
}}
\def\setmid#1#2{{\left\{{#1}:{#2}\right\}}}

\definecolor{applegreen}{rgb}{0.55, 0.71, 0.0}

\newcommand{\jieaoserious}[1]{{\color{black}#1}}

\newcommand{\set}[1]{\left\{#1\right\}}

\title[Projective models for Hilbert squares of $K3$ surfaces]{Projective models for Hilbert squares of $K3$ surfaces}

\author[Á.D.~Ríos~Ortiz, A.~Rojas, and J.~Song]{Ángel David Ríos Ortiz, Andrés Rojas, and Jieao Song}

\begin{document}

\begin{abstract}
For a very general polarized $K3$ surface $S\subset \bP^g$ of genus $g\ge 5$, we study the linear system on the Hilbert square $S^{[2]}$ associated to quadrics in $\bP^g$ that contain $S$. We prove its very ampleness for $g\geq 7$. In the cases of genus 7 or 8, we describe in detail the projective geometry of the corresponding embedding by making use of the Mukai model for $S$. In both cases, it can be realized as a degeneracy locus on an ambient homogeneous space, in a strikingly similar fashion.
In consequence, we give explicit descriptions of its ideal and syzygies.
Furthermore, we extract new information on the locally complete families, in a first step towards the understanding of their projective geometry.
\end{abstract}

\address{\'Angel David R\'ios Ortiz: Université Paris Cité and Sorbonne Université, CNRS, IMJ-PRG \hfill\newline\texttt{}
\indent F-75013 Paris, France}
\email{{\tt riosortiz@imj-prg.fr}}

\address{Andrés Rojas: Departament de Matemàtiques i Informàtica, Universitat de Barcelona \hfill \newline\texttt{}
\indent Gran Via de les Corts Catalanes 585,
08007 Barcelona, Spain}
\email{{\tt andresrojas@ub.edu}}

\address{Jieao Song: Dipartimento di Matematica
``Federigo Enriques'',
Università degli Studi di Milano \hfill \newline\texttt{}
\indent Via Cesare Saldini 50, 20133 Milano, Italy}
\email{{\tt jieao.song@unimi.it}}

\maketitle

\setcounter{tocdepth}{2}

\section{Introduction}
Hyperkähler manifolds, stemming from the Beauville--Bogomolov decomposition theorem as one of the building blocks for manifolds with trivial canonical bundle \cite{beauville}, have attracted considerable attention in the last decades.
In the context of projective hyperkähler manifolds, one of the most challenging problems is to explicitly describe locally complete families. It is thus of compelling interest to understand linear systems on hyperkähler manifolds, their positivity properties (such as very ampleness), behavior of higher syzygies and the geometry of the corresponding projective models.

In the 2-dimensional case, where hyperkähler manifolds are nothing but $K3$ surfaces, Mukai described in a series of influential works \cite{mukai-models,mukai-genus18-20,mukai-genus11,mukai-genus13,mukai-genus16} projective models for general polarized $K3$ surfaces of low degree.
On the other hand, Saint-Donat's classical results \cite{SaintDonat} characterize very ampleness numerically, and describe the equations defining $K3$ surfaces in arbitrary degree. Furthermore, Voisin's groundbreaking proof of generic Green's conjecture for canonical curves \cite{voisin1,voisin2} determines the entire Betti diagram of a $K3$ surface of Picard rank one (embedded in projective space via the primitive linear system). 

For higher-dimensional hyperkähler manifolds, only a few explicit descriptions of locally complete families are known \cite{beauville-donagi, iliev-ranestad, ogrady-epw, debarre-voisin,twistedcubics,epwcubes} (and \cite{stability-families,perry-pertusi-zhao} as well).
Darkness also predominates in the study of equations and syzygies, in contrast to the case of abelian varieties, where, starting with the pioneering work of Mumford \cite{mumford-equations}, remarkable results (e.g.~\cite{koizumi, kempf, pareschi, pareschi-popa}) have been obtained in the last 60 years. Nevertheless, with the only evidence of $K3$ surfaces, many of these results are expected to admit parallels in the hyperkähler setting and thus constitute a source of inspiration (see \cite{ogrady-theta} as an instance of this line of thought for manifolds of generalized Kummer type).

In the sequel, we will restrict ourselves to hyperkähler fourfolds of $K3^{[2]}$-type (namely, those that are deformations of the Hilbert square of a $K3$ surface).
Following \cite{debarre}, we denote by $\cM_{2d}^{(\gamma)}$ the moduli space of polarized fourfolds $(X, H)$ of $K3^{[2]}$-type having square $q(H)=2d$ and divisibility $\div(H)=\gamma$ with respect to the Beauville--Bogomolov--Fujiki quadratic form $q$ on $H^2(X,\bZ)$.%
\footnote{Note that the space $\cM_{2d}^{(\gamma)}$ is non-empty and irreducible of dimension $20$ if and only if $\gamma=1$, or $\gamma=2$ and $2d\equiv 6\pmod 8$ (see \autoref{sec:hyperkähler}). In consequence, when $2d\not\equiv 6\pmod 8$, we will simply speak of a polarization of square $2d$ without specifying the divisibility, which is always $1$.}
In the absence of information for a general element in $\cM_{2d}^{(\gamma)}$, it is natural to specialize to actual Hilbert squares of $K3$ surfaces equipped with appropriate geometric polarizations, and study their positivity properties that might be preserved under deformation.

More precisely, as proven by Beauville \cite{beauville}, for any $K3$ surface $S$ there is an isomorphism 
\[
\Pic(S^{[2]})\cong\Pic(S)\oplus\bZ\cdot \delta,
\]
where $2\delta$ is the class of the divisor in $S^{[2]}$ parametrizing non-reduced subschemes. If we denote by $L_2\in\Pic(S^{[2]})$ the line bundle induced by $L\in\Pic(S)$, then linear systems of the form $|L_2-\delta|$---sending a pair of points in $S$ to the line they span in $\bP(H^0(L)^\vee)$---give a natural way of testing basepoint-freeness and very ampleness.
For example, it can be shown \cite[Section~3.6]{debarre} that the polarization of a general member of $\cM_{2d}^{(\gamma)}$ with square $2d\ge 6$ (and either divisibility $1$ or $2$) is very ample.
On the other hand, it is never very ample for square~$2$ \cite{ogrady-epw,iliev-manivel}, so this leaves polarizations of square $4$ (that is, general members of $\cM_4^{(1)}$) as the only remaining case where the very ampleness is not decided (see also \cite[Problem~1.4]{IKKR}).

The goal of the present paper is to analyze complete linear systems of the form $|L_2-2\delta|$. Given a very ample line bundle $L$ on a $K3$ surface $S$, $|L_2-2\delta|$ can be identified with the space of quadric hypersurfaces in $\bP(H^0(L)^\vee)$ containing $S$, and the induced map
\begin{equation}\label{intro-linearsystem}
S^{[2]}\dasharrow \bP\left(H^0(S^{[2]},L_2-2\delta)^\vee\right)
\end{equation}
sends a length-2 subscheme $\xi$ to the hyperplane of quadrics containing the line $\langle\xi\rangle$ that it spans.

By specializing to $K3$ surfaces of Picard rank one (in which case $L_2-2\delta$ is basepoint-free for $L^2\geq8$), we will show that the linear systems $|L_2-2\delta|$ display a rich geometry, and serve as a good testing ground for positivity properties on polarized hyperkähler fourfolds. This fits well in the philosophy that, for a fixed square $2d$ and divisibility $1$, the linear systems $L_2-r\delta$ (in genus $g=d+r^2+1$) gain positivity when $r$ increases. Of course, the limitation with this approach is that these linear systems involve very difficult geometry of the embedding $S\subset \bP(H^0(L)^\vee)$, and so (even in the presence of a Mukai model) become unmanageable very quickly.

Our first contribution, concerning the geometry of $L_2-2\delta$ in arbitrary genera, is the following.

\begin{thmIntr}
\label{thm:L-2delta-very-ample}
Let $(S,L)$ be a polarized $K3$ surface of genus $g$ with $\Pic(S)=\bZ\cdot L$. Then:
\begin{enumerate}
    \item\label{thm:L-2delta-very-ample-1} If $g\geq7$, the line bundle $L_2-2\delta$ on $S^{[2]}$ is very ample.

    \item\label{thm:L-2delta-very-ample-2} If $g\geq6$, then $(S^{[2]},L_2-2\delta)$ satisfies Le Potier's strange duality with respect to the pair $(\cM(2,L,2),H)$, where $H\coloneqq\theta(1,0,-1)$.
\end{enumerate}
\end{thmIntr}

As an immediate consequence of \autoref{thm:L-2delta-very-ample}.\eqref{thm:L-2delta-very-ample-1} for $g=7$, we can settle very ampleness for the last remaining case for general hyperkähler fourfolds of $K3^{[2]}$-type.

\begin{corIntr}
Let $(X, H)\in \cM_4^{(1)}$ be a general polarized hyperkähler fourfold of $K3^{[2]}$-type with square $4$.
Then $H$ is very ample.
\end{corIntr}

On the other hand, Le Potier's strange duality conjecture is a statement about duality of sections of determinant line bundles on moduli spaces of sheaves. It has been solved in a wide range of cases (where both moduli spaces parametrize sheaves of rank $\geq2$) for generic $K3$ surfaces, see \cite{marian-oprea}. 
In the case $6\leq g\leq 8$, \autoref{thm:L-2delta-very-ample}.\eqref{thm:L-2delta-very-ample-2} was established by O'Grady \cite{ogrady} through a good understanding of the quadrics containing the Mukai model of $S$. In our case, well known results of Green \cite{green} on the rank-4 quadrics containing a canonical curve enable us to generalize the argument to all $g\geq6$.

Next we study the projective geometry of $S^{[2]}$ under the linear system $|L_2-2\delta|$ in the cases $g=7$ and $g=8$,
by making explicit use of their Mukai models. We observe that these two cases share a strikingly similar picture,
where the Hilbert square can be realized as a corank-$2$ degeneracy locus on an ambient homogeneous space of dimension 8.
With the help of the \emph{Gulliksen--Negård complex} \cite{gulliksen-negard} resolving the ideal of $S^{[2]}$ in the ambient homogeneous space, this construction allows us to give a concrete description of the homogeneous ideal and the syzygies of the Hilbert square in the projective space.

Recall from \cite{mukai-models} that a general $K3$ surface $S$ of genus $7$ is a $\bP^7$-linear section of an orthogonal Grassmannian $\OGr(5, V_{10})\subset\bP^{15}$. In particular, there is a natural isomorphism $V_{10}\cong H^0(S^{[2]},L_2-2\delta)$ for the quadrics in $\bP^7$ containing $S$. Note also that, by construction, $\bP(V_{10})$ is canonically equipped with a smooth quadric $Q$ yielding an isomorphism $\bP(V_{10})\cong \bP(V_{10}^\vee)$. In this way, the map in \eqref{intro-linearsystem} gets identified with a closed immersion
\[
S^{[2]}\into \bP(V_{10})=\bP^9
\]
which turns out to factor through the quadric $Q$.

\begin{thmIntr}\label{intro-thmC}
Let $(S,L)$ be a polarized $K3$ surface of genus $7$ with $\Pic(S)=\bZ\cdot L$. Then:
\begin{enumerate}
    \item\label{intro-thmC-1} $S$ induces a morphism of rank-$8$ vector bundles
    \[
        \varphi\colon \cS_+\to \cO^{\oplus 8}
    \]
    on the $8$-dimensional quadric $Q$ in $\bP^9$, where $\cS_+$ is one of the spinor bundles. The embedded $S^{[2]}$ in $\bP^9$ can be identified with the rank-$6$ degeneracy locus $D_6(\varphi)$.

    \item The closed immersion $S^{[2]}\into \bP^{9}$ is normal in degree $d$ for every $d\ge 3$.
    
    \item The homogeneous ideal of $S^{[2]}$ in $\bP^9$ is generated by the quadric $Q$ and $65$ quartics.

    \item If $(S,L)$ is general among $K3$ surfaces of Picard rank one (more precisely, when the singular locus of $D_7(\varphi)$ coincides with $S^{[2]}=D_6(\varphi)$), then the minimal graded free resolution for the ring of sections $R\coloneqq\bigoplus_{d}H^0(S^{[2]}, \cO(d))$ can be described by the Betti diagram \eqref{eq:S2-g7-betti-diagram}.
\end{enumerate} 
\end{thmIntr}

Similarly, recall from \cite{mukai-models} that a general $K3$ surface $S$ of genus $8$ is a $\bP^8$-linear section of the Grassmannian $\Gr(2,V_6)\subset \bP(\bw2 V_6)$. In this case, quadrics in $\bP^8$ containing $S$ admit a natural identification $\bw2 V_{6}\cong H^0(S^{[2]},L_2-2\delta)$, so that \eqref{intro-linearsystem} reads as a closed immersion
\[
S^{[2]}\into \bP(\bw2 V_6^\vee)=\bP^{14}
\]
which turns out to factor through the Grassmannian $\Gr(2,V_6^\vee)$.

\begin{thmIntr}\label{intro-thmD}
Let $(S,L)$ be a polarized $K3$ surface of genus $8$ with $\Pic(S)=\bZ\cdot L$. Then:
\begin{enumerate}
    \item\label{intro-thmD-1} $S$ induces a morphism of rank-$6$ vector bundles
        \[
        \varphi\colon \bw2\cQ^\vee\to \cO^{\oplus 6}
        \]
    on the Grassmannian $\Gr(2,V_6^\vee)\subset\bP^{14}$, where $\cQ$ is the universal quotient bundle.
    The embedded $S^{[2]}$ in $\bP^{14}$ can be identified with the rank-$4$ degeneracy locus $D_4(\varphi)$.

    \item The closed immersion $S^{[2]}\into \bP^{14}$ is projectively normal, with homogeneous ideal generated by $15$ quadrics and $55$ cubics.

    \item Part of the Betti diagram of the coordinate ring of $S^{[2]}$ is given in~\eqref{eq:S2-g8-betti-diagram}.
\end{enumerate}
\end{thmIntr}

Upon completion of this manuscript, we were informed that \autoref{intro-thmC}.\eqref{intro-thmC-1} and \autoref{intro-thmD}.\eqref{intro-thmD-1} were discussed in Benedetti's doctoral dissertation \cite[Section 3.1]{benedetti}.

After \autoref{intro-thmC} and \autoref{intro-thmD}, one is then naturally led to consider the analogous questions on the homogeneous ideal and syzygies for a general deformation of the pair $(S^{[2]},L_2-2\delta)$.
In the case of genus $7$, after performing some deformation-theoretic arguments, we prove that the projective model of a general deformation of $(S^{[2]},L_2-2\delta)$ is \emph{not} contained in a quadric hypersurface. This allows us to deduce its full minimal graded resolution.
Our main results about general members of $\cM_4^{(1)}$ and $\cM_6^{(1)}$, a first step towards the understanding of their projective geometry, read as follows:
\begin{thmIntr}\label{intro-thmdef}\leavevmode
\begin{enumerate}
\item\label{intro-thmdef-1} For a general polarized hyperkähler fourfold $(X, H)\in\cM_4^{(1)}$, the embedding
\[
X\into \bP(H^0(X,H)^\vee)=\bP^9
\]
is projectively normal.
Furthermore, its minimal graded free resolution is described by the Betti diagram \eqref{eq:betti-diagram}. In particular, the homogeneous ideal of $X$ in $\bP^9$ is generated by $10$ cubics (with no linear syzygies) and $20$ quartics.
\item\label{intro-thmdef-2} For a general polarized hyperkähler fourfold $(X, H)\in \cM_6^{(1)}$, the embedding
\[
X\into \bP(H^0(X,H)^\vee)=\bP^{14}
\]
is projectively normal. The homogeneous ideal of $X$ in $\bP^{14}$ is generated by $15$ quadrics and at most $55$ cubics (one expects $20$ cubics).
\end{enumerate}
\end{thmIntr}

We would like to point out that it can be proved that a general fourfold in $\cM_4^{(1)}$ (respectively in $\cM_6^{(1)}$) is the scheme-theoretical intersection of the $10$ cubics (respectively of the $15$ quadrics) alone,%
\footnote{Note that the extra equations of higher degree are needed to generate the (saturated) homogeneous ideal.}
which are the polars of a unique hypersurface of degree $4$ (respectively of degree $3$).
Manifestly, these hypersurfaces bear a striking resemblance to the classical Coble hypersurfaces for abelian varieties \cite{beauville-coble}. A detailed investigation of these \emph{Coble type hypersurfaces} appears in our subsequent work \cite{coble-type-hypersurfaces}. 

Let us also note that this is in remarkable contrast with the varieties of lines of cubic fourfolds (general elements of $\cM_6^{(2)}$), where the projective normality can be checked via the Koszul complex; in this case, the 15 quadrics are nothing but the Plücker quadrics containing the Grassmannian $\Gr(2,6)$, and so are far from cutting out the hyperkähler fourfold.

Finally, it is worth noting that the minimal resolution of the coordinate ring of a general $(X,H)\in \cM_4^{(1)}$ behaves exactly as the Hilbert function predicts, in the sense that there are no unexpected syzygies.
This behavior suggests that the syzygies for a general polarized hyperkähler fourfold of $K3^{[2]}$-type with divisibility $1$ might follow their surface counterparts and be ``as simple as possible''.
This leads to some concrete conjectural descriptions, which shall be addressed in a subsequent work.
They are included at the end of this paper, together with some other natural observations and remarks arising from our results.

\subsection*{Structure of the paper}
We recall some preliminary results in \autoref{sec:prelim}, and notably in \autoref{sec:prelim-lowrank} the description of the locus of quadrics of low rank containing a $K3$ surface.
The description is crucial for the proof of \autoref{thm:L-2delta-very-ample}, which will be presented in \autoref{sec:proof-thm-A}.
Then we study the two cases of Hilbert squares of $K3$ surfaces in genus $7$ and $8$, as well as their general deformations, in \autoref{sec:genus-7} and \ref{sec:genus-8}, respectively.
Finally, in \autoref{sec:questions} we raise several questions concerning the geometry of $(S^{[2]},L_2-2\delta)$ in arbitrary genus, as well as the syzygies of their general deformations.

\subsection*{Acknowledgments}
We are grateful to Chiara Camere, Martí Lahoz, Federico Moretti, and Benedetta Piroddi  for useful conversations. 
We also thank Grzegorz and Michał Kapustka for informing us about the results contained in \cite{benedetti}, Vladimiro Benedetti for follow-up discussions, and an anonymous referee for her or his suggestions.

Despite all proofs being done in the ``old-fashioned way'', the usage of \emph{free and open-source} computer algebra systems {\sc Macaulay2} \cite{macaulay2}, {\sc Singular} \cite{singular}, and {\sc Sage} \cite{sagemath} is instrumental in this project.
We would like to thank the developers for their efforts.

Ríos Ortiz was supported by the European Research Council (ERC) under the European Union’s Horizon 2020 research and innovation programme (ERC-2020-SyG-854361-HyperK). Rojas was supported by a Beatriu de Pinós postdoctoral fellowship (no.~2023 BP 00054), the Spanish MICINN project PID2023-147642NB-I00, and the ERC Advanced Grant SYZYGY. Song was supported by a BMS Dirichlet postdoctoral fellowship and the PRIN2022 research grant 2022PEKYBJ ``Symplectic varieties: their interplay with Fano manifolds and derived categories''.

\section{Preliminaries}
\label{sec:prelim}

\subsection{Generalities on hyperkähler manifolds}
\label{sec:hyperkähler}
We start by recalling some basic properties on hyperkähler manifolds. An excellent introduction to the subject can be found in \cite{debarre}.

A simply connected compact Kähler manifold $X$ is called a \emph{hyperkähler manifold} if $H^0(X,\Omega_X^2)$ is
generated by a nowhere degenerate holomorphic $2$-form, in other words, a symplectic form.
In this case the dimension of $X$ is necessarily even.
Let $X$ be a hyperkähler manifold of dimension $2m$.
There exists a canonical integral primitive quadratic form $q\coloneqq q_X$ on $H^2(X,\bZ)$ called the \emph{Beauville--Bogomolov--Fujiki form} of $X$.

Let $(X,H)$ be a polarized hyperkähler manifold, where $H$ is an ample class on $X$. We have the following two invariants of $H$:
\begin{itemize}
\item The \emph{square} $q(H)$ with respect to the quadratic form;
\item The \emph{divisibility} $\div(H)$, which is the generator of the subgroup $q(H, H^2(X, \bZ))$ of $\bZ$.
\end{itemize}
The theory for moduli spaces of polarized hyperkähler manifolds is very well understood via the global Torelli theorem, thanks to the works of Verbitsky~\cite{verbitsky} and Markman~\cite{markman-survey}.

By \cite{huybrechts} there exists a polynomial $\mathrm{RR_X}(q)\in\bQ[q]$ called the \emph{Riemann--Roch polynomial} of $X$, such that for any line bundle $L\in \Pic(X)$ we have
\[
\chi(X,L) =\mathrm{RR}_X(q(L)).
\]
If $L$ is moreover ample, then thanks to the Kodaira vanishing theorem, the Riemann--Roch polynomial computes the number of global sections $h^0(X,L)$.

A hyperkähler manifold $X$ is said to be \emph{of $K3^{[m]}$-type} if it is deformation equivalent to the Hilbert scheme of $m$ points on a $K3$ surface.
For $K3^{[m]}$-type, we have the following description of the Riemann--Roch polynomial (see \cite[Lemma~5.1]{egl})
\[
\mathrm{RR}_{K3^{[m]}}(q) = \binom{\frac12 q + m+1}{m}.
\]
In particular, if $(X,H)$ is of $K3^{[2]}$-type with $H$ very ample and of square $2d$, assuming moreover that $X$ is projectively normal when embedded in $\bP^n$, where $n=\binom{
d+3}{2}-1$, we may compute the Hilbert function
\begin{equation}
\label{eq:hilbert-function}
\dim I_e = h^0(\bP^n,\cO(e)) - h^0(X, H^e)= \binom{n+e}{e} - \binom{de^2+3}{2}
\end{equation}
where $I_e \coloneqq h^0(\bP^n, \cI_{X/\bP^n}(e))$. Notably, $\dim I_e$ does not depend on the divisibility of $H$.

Recall from the introduction that we denote by $\cM_{2d}^{(\gamma)}$ the moduli space of polarized hyperkähler manifolds $(X, H)$ of $K3^{[2]}$-type with square $q(H)=2d$ and divisibility $\div(H)=\gamma$.
The space $\cM_{2d}^{(\gamma)}$ is non-empty if and only if $\gamma=1$, or $\gamma=2$ and $2d\equiv 6\pmod 8$ \cite[Theorem~3.5]{debarre}.
In all cases, it is irreducible of dimension $20$.

\subsection{Syzygies and Green's duality theorem}
We recall some important notions in the study of syzygies for a projective variety.
We will state the results in a limited generality that will be sufficient for the context of this paper.
The definitive reference on the subject is of course Green's influential paper \cite{greenkoszul}, where one can find the original statements in their full generality.

Let $X$ be a projective variety and let $\varphi_H \colon X\into \bP V^\vee$ be the embedding defined by a very ample line bundle $H$ on $X$, where $V\coloneqq H^0(X, H)$. We denote by $S_X\coloneq \Sym^\bullet V=\Sym^\bullet H^0(X,H)$ the symmetric algebra and by $R_X\coloneqq \bigoplus_{d\ge 0} H^0(X,H^d)=\bigoplus_{d\ge 0} R_d$ the \emph{ring of sections}.%
\footnote{In the case where $X$ is projectively normal, $R_X$ is simply the quotient $S_X/I_X$ and is usually referred to as the \emph{coordinate ring} of $X$. We adopt the terminology to distinguish the two in cases where $X$ is not projectively normal. \jieaoserious{Another appropriate name would be the \emph{ring of power}.}}
Clearly $R_X$ admits the structure of a graded $S_X$-module via the natural map $m:S_X\to R_X$. The embedding $\varphi_H$ is said to be \emph{$d$-normal} if $m$ is surjective in degree $d$; it is said to be \emph{projectively normal} if it is $d$-normal for every $d$.

By Hilbert's syzygy theorem, there is a unique (up to isomorphism) minimal graded free resolution
\[
\dots\to F_2\to F_1\to F_0\to R_X\to 0,
\]
where
\[
F_i=\bigoplus_j S_X(-i-j)^{b_{i,j}}
\]
are free $S_X$-modules with shifted degrees.
The numbers $b_{i,j}$ are called the \emph{Betti numbers} and are invariants for the embedding $\varphi_H$.
It is customary to arrange them into the following \emph{Betti diagram}, which provides a convenient way to visualize the properties of the embedding.
\[
\begin{array}{r|ccccc}
  & 0       & 1       & 2       & 3       & \dots \\
\hline                         
0 & b_{0,0} & b_{1,0} & b_{2,0} & b_{3,0} & \dots \\
1 & b_{0,1} & b_{1,1} & b_{2,1} & b_{3,1} & \dots \\
2 & b_{0,2} & b_{1,2} & b_{2,2} & b_{3,2} & \dots \\
\vdots & \vdots & \vdots & \vdots & \vdots & \ddots \\
\end{array}
\]
The Betti numbers can be effectively computed by means of the \emph{Koszul cohomology groups}. Namely, one may consider the Koszul complex
\[
\cdots \to \bw{i+1} V\otimes R_{j-1}\to  \bw{i} V\otimes R_{j} \to  \bw{i-1} V\otimes R_{j+1}\to \cdots
\]
and let $\cK_{i,j}\coloneqq \cK_{i,j}(X,H)$ be the cohomogology group of the term in the middle.
Then we have $\dim \cK_{i,j}=b_{i,j}$ \cite[Theorem~1.b.4]{greenkoszul}.

Concerning varieties with trivial canonical bundle, we have the following result (see [\loccit, Corollary~2.c.10]) that predicts some symmetry of the Betti diagram.
\begin{thm}[Green's duality theorem]
\label{thm:green-duality}
Given $\varphi_H:X\hookrightarrow \bP^n$ as above for a smooth projective variety with trivial canonical
bundle $\omega_X\cong \cO_X$, we have for all $j\ge \dim X+1$
\[
\cK_{i,j}(X,H)^\vee \cong \cK_{\mathrm{codim}(X/\bP^n) -i, \dim X + 1-j}(X, H),
\]
and therefore the symmetry
\[
b_{i,j} = b_{\mathrm{codim}(X/\bP^n) -i, \dim X + 1-j}.
\]
Moreover, if $h^{\dim X-1}(\cO_X)=0$,%
\footnote{By the assumption that $\omega_X\cong \cO_X$, this is equivalent to
$h^1(\cO_X)=0$.}
then the duality holds for all $j\ge \dim X$.
\end{thm}

\subsection{Moduli spaces of sheaves on \texorpdfstring{$K3$}{K3} surfaces} 
We recall some basic aspects of moduli spaces of stable sheaves on $K3$ surfaces and their symplectic structure, a subject studied extensively by several authors, including Mukai, O'Grady, Huybrechts, and Yoshioka. The reader may consult \cite[Chapter 10]{lecturesK3} (and the references therein) for further details. We will assume, here and throughout the paper, that $(S,L)$ is a polarized $K3$ surface with $\Pic(S)=\bZ\cdot L$.

Let $\widetilde{H}(S,\bZ)$ denote the (even) cohomology ring of $S$. The pure weight-two Hodge structure on $H^2(S,\bZ)$ can be extended to $\widetilde{H}(S,\bZ)$ by requiring the degree 0 and 4 parts to be algebraic. The symmetric bilinear form
\[
\langle u,v \rangle \coloneqq \int_S u_2\wedge v_2 - \int_S (u_0\wedge v_4 + u_4\wedge v_0)   
\]
equips $\widetilde{H}(S,\bZ)$ with the structure of a lattice, called the \emph{Mukai lattice} of $S$. 

Given a coherent sheaf $F$ on $S$, we define its \textit{Mukai vector} as
\[
v(F)\coloneqq\mathrm{ch}(F)\cdot\sqrt{\mathrm{td}(S)}=\left(\rk(F), c_{1}(F), \ch_2(F)+\rk(F)\right) \in \widetilde{H}(S,\bZ). 
\]
The Hirzebruch--Riemann--Roch formula implies that for coherent sheaves $F,G$ on $S$ we have
\[
\chi(E,F)=-\langle v(E),v(F)\rangle.
\]

Let \(v=(r,cL,s) \in \widetilde{H}^{1,1}(S, \bZ)\) be a primitive vector with $r>0$. We denote by \(\cM(v)\) the moduli space of \(L\)-Gieseker stable sheaves on \(S\) with Mukai vector \(v\). It is known that \(\cM(v)\) is non-empty exactly when $v^2\geq -2$; in that case, \(\cM(v)\) is a projective hyperkähler manifold of $K3^{[m]}$-type, where $m=\frac{1}{2}(v^2+2)$.

The Beauville--Bogomolov--Fujiki form for $\cM(v)$ can be explicitly described as follows. There exists a quasi-universal family of sheaves in $\cM(v)$ of similitude $\rho\in\bZ_{>0}$, that is, a sheaf $\cE$ on $S\times \cM(v)$ flat over $\cM(v)$ such that $\cE|_{S\times [F]}\cong F^{\oplus \rho}$ for every $[F]\in \cM(v)$. For $v^2\geq2$, letting $\theta_{v}\colon \widetilde{H}(S,\bZ)\to H^2(\cM(v),\bZ)$ be defined by
\[
\theta_{v}((\alpha_0,\alpha_2,\alpha_4))\coloneqq \frac{1}{\rho}\left[ \mathrm{pr}_{\cM(v),!} \left(\mathrm{ch}(\cE)\smile\mathrm{pr}_S^{*}(\sqrt{\mathrm{td}(S)}\smile(\alpha_0,-\alpha_2,\alpha_4))\right)\right]_2,
\]
it was proved in \cite{ogrady-hodge,yoshioka-reflections} that restriction to $v^\perp$ yields a Hodge isometry
\begin{equation}\label{thetaiso}
 \theta_{v}\colon v^{\perp}\to H^2(\cM(v),\bZ).
\end{equation}

For $m\geq2$, the moduli space $\cM(1,0,-m+1)$ parametrizes ideal sheaves $\cI_Z$ where $Z$ is a length-$m$ 0-dimensional subscheme of $S$, namely $S^{[m]}\coloneqq \cM(1,0,-m+1)$ is the Hilbert scheme of $m$ points on $S$. In that case,  as proved in \cite{beauville}, \eqref{thetaiso} reads as a Hodge isometry \[\left(H^2(S^{[m]},\bZ),q\right)\cong H^2(S,\bZ)\overset{\perp}{\oplus}\bZ\cdot\delta
\]
where $\delta$ is a $(1,1)$-class of square $-2(m-1)$. In particular
$\Pic(S^{[m]})\cong\Pic(S)\oplus\bZ\cdot\delta$,  
where $2\delta$ is the class of the divisor $\Delta$ of non-reduced subschemes. For $L\in \Pic(S)$, the corresponding line bundle on $S^{[m]}$ is obtained by symmetrization and will be denoted by $L_m\in\Pic(S^{[m]})$.
\begin{rem}\label{explicit-thetaiso}
Let $\cI$ denote the universal ideal sheaf on $S\times S^{[m]}$, and let $\pi,\psi$ denote the projections to $S^{[m]}$ and $S$ respectively. Write $v=(1,0,-m+1)$. If $v'\in v^\perp$ is a $(1,1)$-class with $\rk(v')>0$, or $\rk(v')=0$ and $c_1(v')$ effective, then $\theta_v(v')\coloneqq\det R\pi_!(\cI\otimes \psi^*F)^{-1}$ for any sheaf $F$ of Mukai vector $v'$. It follows that $\theta_{v}(1,0,m-1) = -\delta$ and $\theta_{v}(0,L,0) = L_m$.
\end{rem}

In the sequel, we will focus on the case $m=2$. Consider the diagram
\begin{equation}\label{eq:universallenght2subscheme}
        \begin{tikzcd}
            &B\ar[ld, "\pi"']\ar[rd, "\psi"]\coloneqq\setmid{(p,\xi)\in S\times S^{[2]}}{p\in \xi}\\
            S^{[2]}&& S
        \end{tikzcd}   
\end{equation}
where $B$ is the universal length-$2$ subscheme (namely, $\pi$ is a double cover branched along $\Delta$). For any sheaf $F$ on $S$ we can associate a \emph{tautological sheaf} $F^{[2]}$ on $S^{[2]}$ by the rule
\[
F^{[2]} \coloneqq \pi_*\psi^*F.
\]
If $F$ is a vector bundle of rank $r$, then the corresponding tautological sheaf is a vector bundle of rank $2r$, whose fiber at a point $\xi\in S^{[2]}$ corresponds to $H^0(S,F\otimes\cO_\xi)$. 

Recall that if $\Pic(S)=\bZ\cdot L$ and $g\geq5$, then $S\subset \bP(H^0(L)^\vee)=\bP^g$ is cut out by quadrics \cite{SaintDonat}. There is a natural identification $H^0(S^{[2]},L_2-2\delta)\cong H^0(\bP^g,\cI_{S/\bP^g}(2))$ with the space of quadratic equations for $S$ in $\bP^g$, and the resulting morphism
\[
\varphi\colon S^{[2]}\to |L_2-2\delta|^\vee
\]
associates to each $\xi\in S^{[2]}$ the hyperplane of quadric hypersurfaces in $\bP^g$ containing both $S$ and the line spanned by $\xi$. Note that the morphism $\varphi$ is indeed well-defined, as $S$ is cut out by quadrics and contains no line. We have a canonical short exact sequence
\[
0\to K \to H^0\left(\cI_{S/\bP}(2)\right)\otimes\cO_{\bP}\to \cN^\vee_{S/\bP}(2)\to 0,
\]
where $K$ is an extension of $\cI_{X/\bP}^2(2)$ by $\ker\bigl(H^0\left(\cI_{S/\bP}(2)\right)\otimes\cO_{\bP}\twoheadrightarrow \cI_{S/\bP}(2)\bigr)$. It follows that $H^0(K)=0$, hence there is a canonical inclusion
\begin{equation}\label{quadrics-conormal}
    H^0\left(\cI_{S/\bP}(2)\right)\into H^0\left(\cN^\vee_{S/\bP}(2)\right)
\end{equation}
such that, for a given quadric $Q\in H^0\left(\cI_{S/\bP}(2)\right)$, the zero locus of the induced global section of the twisted conormal bundle $\cN^\vee_{S/\bP}(2)$ equals $S\cap\Sing(Q)$.

\subsection{Quadrics of low rank containing \texorpdfstring{$K3$}{K3} surfaces}
\label{sec:prelim-lowrank}

In this subsection, we recall a few aspects of the  locus of quadrics of rank $\leq 6$ in $|I_S(2)|\coloneqq \bP H^0(\cI_{S/\bP^g}(2))$. As it is well known to the experts, this locus can be described in terms of stable rank $2$ sheaves on the $K3$ surface $S$. We provide some details on the description of this locus following the lines of \cite[Section~5]{ogrady}, where O'Grady focuses on a particular component $Y_0$.

Consider a pair $(E,V)$, for a rank-$2$ stable sheaf $E\in\cM(2,L,\ell+2)$ (with $0\leq \ell \leq \lfloor\frac{g}{2}\rfloor-2$) and a 4-dimensional subspace $V\subset  H^0(E)$. Associated to it we have a rational map
\[
\rho_{E,V}\colon S\dasharrow \Gr(2,V^\vee)\subset \bP(\bw2 V^\vee),\quad p\mapsto \left(V\cap H^0(E\otimes\cI_p)\right)^\perp.
\]
Note that $\rho_{E,V}$ is undefined exactly at the finite set of points $p\in S$ where either $E$ is not locally free, or $E$ is locally free but $V\cap H^0(E\otimes\cI_p)$ has codimension $<2$ in $V$.
Moreover, over the open subset $U\subset S$ where $\rho_{E,V}$ is defined, it is straightforward to check that $\rho_{E,V}^*\cU^\vee\cong E|_U$, where $\cU$ is the universal subbundle of rank $2$.

The natural map $\lambda_{E,V}\colon\bw2 V\into \bw2 H^0(E)\overset\det\to H^0(L)$ yields a factorization
\begin{equation}
\label{eq:diagram}
\begin{tikzcd}
S\ar[rd,dashed,"\rho_{E,V}"']\ar[r,hookrightarrow]&\bP\left(H^0(L)^\vee\right)\ar[r,dashed]&\bP\left(\mathrm{Im}(\lambda_{E,V})^\vee\right)
\ar[d,hookrightarrow,"\bP(\lambda_{E,V}^\vee)"]  \\
&\Gr(2,V^\vee)\ar[r,hookrightarrow]&\bP\left(\bw2 V^\vee\right)
\end{tikzcd}
\end{equation}

where the rational map in the first row is the linear projection from the annihilator subspace $\bP\left(\mathrm{Im}(\lambda_{E,V})^\perp\right)\subset \bP\left(H^0(L)^\vee\right)$.
In other words, $\bP(\mathrm{Im}(\lambda_{E,V})^\vee)$ is the linear span of the image $\rho_{E,V}(S)$.

The Grassmannian $\Gr(2,V^\vee)\subset \bP(\bw2V^\vee)$ is a smooth quadric.
Therefore, as long as it does not contain the linear span of $\rho_{E,V}(S)$, the pullback defines a quadric $Q_{E,V}\subset \bP\left(H^0(L)^\vee\right)$ containing $S$.
And this is indeed the case by the following lemma.

\begin{lem}[{\cite[Lemma 5.4]{ogrady}}] \label{OG-kernel}
$\ker(\lambda_{E,V})$ does not contain decomposable tensors.
In particular $\dim\ker(\lambda_{E,V})\le 1$,
that is, $\bP(\mathrm{Im}(\lambda_{E,V})^\vee)$ has codimension $\le 1$ in $\bP(\bw2 V^\vee)$.
\end{lem}

By construction, the singular locus of $Q_{E,V}$ is $\bP\left(\mathrm{Im}(\lambda_{E,V})^\perp\right)\subset \bP\left(H^0(L)^\vee\right)$,
the linear subspace we are projecting from.
Therefore, the quadric has rank either $5$ or $6$.

Performing this construction in families, we obtain for every $\ell$ ($0\leq \ell \leq \lfloor\frac{g}{2}\rfloor-2$) a natural morphism
\begin{equation}
\label{eq:def-relGrass}
\psi_{\ell}\colon \cG_\ell\to \setmid{Q\in|I_S(2)|}{\rk(Q)\leq 6}\subset |I_S(2)|,\quad(E,V)\longmapsto Q_{E,V}
\end{equation}
where $\cG_\ell$ is the relative Grassmannian with fiber $\Gr(4,H^0(E))$ over $E\in\cM(2,L,\ell+2)$.
Note that $\cG_\ell$ are all irreducible of dimension $2g-8$.
We will denote the image by $Y_\ell\coloneqq\im(\psi_\ell)\subset |I_S(2)|$.

\begin{rem}
If $\ell=0$, then $\psi_{0}$ is the resolution of the rational map $\cM(2,L,2)\dasharrow |I_S(2)|$ studied by O'Grady in \cite{ogrady}. Indeed, $\cG_0$ is the blow-up of $\cM(2,L,2)$ along the indeterminacy locus, namely the Brill--Noether locus $\setmid{E\in\cM(2,L,2)}{h^1(E)>0}$.
\end{rem}

The following lemma shows that, via the morphisms $\psi_\ell$, vector bundles of rank $2$ serve as a bridge between quadrics of rank $\leq 6$ in $|I_S(2)|$ and secant lines to $S$. This will be essential in our analysis of the linear system $|L_2-2\delta|$ in $S^{[2]}$:

\begin{lem}\label{containmentcrit}
Let $(E,V)\in\cG_\ell$ with $E$ a vector bundle and $0\leq \ell \leq \lfloor\frac{g}{2}\rfloor-2$. For every $\xi\in S^{[2]}$, the following are equivalent:
\begin{enumerate}
    \item\label{containmentcrit-1} The line  $\langle\xi\rangle\subset\bP\left(H^0(L)^\vee\right)$ is contained in the quadric $Q_{E,V}\in |I_S(2)|$.
    \item\label{containmentcrit-2} $V\cap H^0(E\otimes\cI_\xi)\neq0$.
\end{enumerate}
\end{lem}
\begin{proof}
    First assume that $\xi$ intersects the indeterminacy locus of $\rho_{E,V}$ at some point $p\in\Supp(\xi)$.
    Then both conditions \eqref{containmentcrit-1} and \eqref{containmentcrit-2} are satisfied. Indeed, by the commutative diagram \eqref{eq:diagram} $p$ is a singular point of $Q_{E,V}$, which implies $\langle\xi\rangle\subset Q_{E,V}$; on other hand,
    since $V\cap H^0(E\otimes\cI_p)$ has codimension $<2$, the intersection $V\cap H^0(E\otimes\cI_\xi)$ must be non-empty.

    Now we assume that $\xi$ is disjoint from the indeterminacy locus of $\rho_{E,V}$. Denote by $\eta=\rho_{E,V}(\xi)$ the (scheme-theoretic) image of $\xi$ in $\Gr(2,V^\vee)$, and by $\cU$ the universal subbundle in $\Gr(2,V^\vee)$. Observe that for all $s\in V=H^0(\Gr(2,V^\vee),\cU^\vee)$ the zero locus $Z_{\Gr(2,V^\vee)}(s)=\Gr(2,s^\perp)\subset\Gr(2,V^\vee)$ is a 2-plane contained in the quadric $\Gr(2,V^\vee)$.

    If \eqref{containmentcrit-2} holds, namely $\xi\subset Z(s)$ for some nonzero $s\in V$, then we have $\eta\subset Z_{\Gr(2,V^\vee)}(s)\cong \bP^2$ in $\Gr(2,V^\vee)$. Therefore the linear span
    $\langle\eta\rangle\subset \bP(\bw2V^\vee)$ is contained in $\Gr(2,V^\vee)$ and hence the inclusion $\langle\xi\rangle\subset Q_{E,V}$ holds in $\bP\left(H^0(L)^\vee\right)$.

    Conversely, if \eqref{containmentcrit-1} holds, the linear span $\langle\eta\rangle\subset \bP(\bw2V^\vee)$ (which is either a line or a point) is contained in $\Gr(2,V^\vee)$; this implies that all elements in $\langle\eta\rangle$ (regarded as alternating forms on $V$) have a common element $s\in V$ in their kernel. Such an $s$ satisfies $\langle\eta\rangle\subset\Gr(2,s^\perp)$ and  therefore $s\in H^0(E\otimes\cI_\xi)$, which finishes the proof.
\end{proof}

It is a natural question whether the varieties $Y_\ell$ recover all quadrics of rank $\leq 6$ containing $S$.
To this end, we will provide an intrinsic description for $Y_\ell$ in terms of the quadrics they contain.
Consider the following locally closed subset of $|I_S(2)|$:
\[
\Sigma\coloneqq\setmid{Q\in|I_S(2)|}{\rk(Q)=6,\; \Sing(Q)\cap S=\emptyset}.
\]

For $Q\in \Sigma$, the projection from the singular locus of $Q$ defines a morphism $\rho\colon S\to\bP^5$ whose image is contained in a smooth quadric $Q_4\subset\bP^5$.
It is well known that $H^4(Q_4,\bZ)$ is spanned by the classes $P_1$ and $P_2$ of the two types of 2-planes contained in $Q_4$, and $h^2=P_1+P_2$ where $h=c_1(\cO_{Q_4}(1))$.
Hence, there exist $a,b\in \bZ_{\ge0}$ with
\[
[\rho(S)]=a P_1 + b P_2,\,\ a+b=\deg S=2g-2.
\]
Therefore, we have the following invariant of a quadric $Q\in\Sigma$ which is locally constant on $\Sigma$
\[
i(Q)\coloneqq\tfrac{1}{2}\bigl\lvert a-b\bigr\rvert\in\set{0,\dots,g-1}.
\]
We can thus define
\[
\Sigma_\ell\coloneqq\setmid{Q\in \Sigma}{i(Q)=\ell},
\]
and immediately obtain a decomposition
\[
\Sigma=\Sigma_0\sqcup \dots\sqcup \Sigma_{g-1}.
\]

\begin{prop}
\label{prop:sigma-closure}
In the above notation, the following statements hold:
\begin{enumerate}
    \item\label{sigma-bounded} If $\ell>\lfloor\frac{g}{2}\rfloor-2$, then $\Sigma_\ell=\emptyset$.

    \item\label{sigma-closure-is-Y} For any $\ell\in\bigl\{0,...,\lfloor\frac{g}{2}\rfloor-2\bigr\}$, $Y_{\ell}$ equals the closure of $\Sigma_\ell$ in $|I_S(2)|$.

    \item\label{sigma-closure-is-all} $\bigl\{Q\in|I_S(2)|:\mathrm{rk}(Q)\leq 6\bigr\}=\overline{\Sigma}=\bigcup_{\ell=0}^{\lfloor\frac{g}{2}\rfloor -2}Y_\ell$.
\end{enumerate}
\end{prop}
\begin{proof}
Given $Q\in\Sigma$, projection from the singular locus of $Q$ defines a morphism $\rho\colon S\to\bP^5$ whose image is contained in a smooth quadric $Q_4\subset\bP^5$.
Denote by $\cS_+$, $\cS_-$ the two spinor bundles on $Q_4$.
The two types of 2-planes in $Q_4$ are given as the zero locus of a global section of either $\cS_+^\vee$ or $\cS_-^\vee$.
Therefore, up to exchanging $\cS_+$ and $\cS_-$, we may assume that $c_2(\rho^*\cS_-^\vee)=g-1-i(Q)$ and $c_2(\rho^*\cS_+^\vee)=g-1+i(Q)$.
Note that by identifying $Q_4$ with the Grassmannian $\Gr\left(2,H^0(Q_4,\cS_+^\vee)\right)$, we have $\cS_+^\vee\cong\cQ$ and $\cS_-\cong\cU$ where $\cQ$ (resp. $\cU$) is the universal quotient (resp. the universal subbundle) of rank $2$. 

We claim that the pullbacks $\rho^*\cS_+,\rho^*\cS_-$ are slope stable. Indeed, as $\rho(S)$ is non-degenerate in $\bP^5$ we have an inclusion $H^0(Q_4,\cS_+^\vee)\subset H^0(S,\rho^*\cS_+^\vee)$, which in view of the short exact sequence
\[
0\to \rho^*\cS_-\to H^0(Q_4,\cS_+^\vee)\otimes\cO_S\to \rho^*\cS_+^\vee\to 0
\]
implies $H^0(S,\rho^*\cS_-)=0$. Therefore there is no inclusion $L^{\otimes t}\into \rho^*\cS_-$ with $t\geq0$, so $\rho^*\cS_-$ is slope stable. Similarly, $\rho^*\cS_+$ is also slope stable.
Since $v(\rho^*\cS_-)=(2,-L,i(Q)+2)$, the inequality $v(\rho^*\cS_-)^2\geq -2$ reads as $2(i(Q)+2)\leq g$, which proves \eqref{sigma-bounded}.

In order to prove \eqref{sigma-closure-is-Y}, observe that the argument above shows that every $Q\in\Sigma_\ell$ lies in $Y_\ell$; indeed, $Q=\psi_{\ell}(E,V)$ for $E=\rho^*\cS_-^\vee$ with the 4-dimensional subspace $V=H^0(Q_4,\cS_-^\vee)\subset H^0(S,\rho^*\cS_-^\vee)$.
Conversely, for a general $(E,V)\in\cG_{\ell}$ (corresponding to a globally generated vector bundle $E$, and $V\in\Gr(4,H^0(E))$ such that $\bw2V\to H^0(L)$ is injective) the corresponding quadric $\psi_{\ell}(E,V)$ lies in $\Sigma_\ell$. This shows \eqref{sigma-closure-is-Y}.

Now we prove \eqref{sigma-closure-is-all}. The second equality is clear thanks to \eqref{sigma-bounded} and \eqref{sigma-closure-is-Y}. Therefore, it suffices to check that $\Sigma$ is dense in $\bigl\{Q\in|I_S(2)|:\mathrm{rk}(Q)\leq 6\bigr\}$.
Note also that any quadric $Q\in|I_S(2)|$ has rank $\geq5$ by the hypothesis $\Pic(S)=\bZ\cdot L$, see for instance \cite[Claim 5.2]{ogrady}.

First we consider $Q\in|I_S(2)|$ of rank $5$ with $\Sing(Q)\cap S=\emptyset$. Linear projection from $\Sing(Q)$ defines a morphism $\rho\colon S\to\bP^4$ factoring through a smooth quadric $Q_3\subset\bP^4$. Let $\cS$ be the rank-$2$ spinor bundle on $Q_3$. Since $c_1(\cS^\vee)=h$ and $c_2(\cS^\vee)=\frac{1}{2}h^2$ where $h$ is the hyperplane class in $Q_3$ (c.f. \cite[Remark~2.9]{ottaviani}), the pullback $\rho^*\cS^\vee$ to $S$ has Mukai vector $v(\rho^*\cS^\vee)=(2,L,2)$. 
Furthermore, $\rho^*\cS$ is slope stable; this can be argued along the same lines as in the rank-$6$ case, by combining the injection $H^0(Q_3,\cS^\vee)\subset H^0(S,\rho^*\cS^\vee)$ with the short exact sequence
\[
0\to \rho^*\cS \to H^0(Q_3,\cS^\vee)\otimes \cO_S\to \rho^*\cS^\vee \to 0
\]
(pulled back from $Q_3$). Then one easily checks that $Q=\psi_0(E,V)$, where $E=\rho^*\cS^\vee$ and $V$ is the 4-dimensional subspace $H^0(Q_3,\cS^\vee)\subset H^0(S,\rho^*\cS^\vee)$.

If $Q\in|I_S(2)|$ has rank $6$ and $\Sing(Q)\cap S\neq\emptyset$, then we consider the rational map $\rho\colon S\dasharrow\bP^5$ with image contained in a smooth quadric $Q_4\subset \bP^5$. Let $U\subset S$ be the open subset where $\rho$ is defined (namely, $S\setminus U=\Sing(Q)\cap S$ has codimension $2$ on $S$). Pulling back to $U$ the two spinor bundles $\cS_+,\cS_-$ on $Q_4$ defines a short exact sequence of $\cO_U$-modules
\[
0\to \rho^*\cS_-\to H^0(Q_4,\cS_+^\vee)\otimes \cO_U\to \rho^*\cS_+^\vee\to 0.
\]
Pushing forward via $i\colon U\into S$ results in a short exact sequence
\begin{equation}\label{ses-pushforward}
0\to E_-\to H^0(Q_4,\cS_+^\vee)\otimes\cO_S\to E\to 0,
\end{equation}
where $E_-$ is locally free with $E_-|_U= \rho^*\cS_-$, and $E\subset i_*(\rho^*\cS_+^\vee)$ is torsion-free with $E|_U= \rho^*\cS_+^\vee$. 
Observe that $H^0(S,E_-)=H^0(U,\rho^*\cS_-)=0$ (as $H^0(Q_4,\cS_+^\vee)\subset H^0(U,\rho^*\cS_+^\vee)$ by non-degeneracy of $\rho(S)$ in $\bP^5$), which as above implies the slope stability of $E_-$. Similarly (exchanging the role of $\cS_+$ and $\cS_-$, which dualizes \eqref{ses-pushforward}) one gets slope stability of $E$. 

Moreover, in view of \eqref{ses-pushforward} we have $\ch_2(E)\geq0$ or $\ch_2(E_-)\geq0$. In the first case, we have $Q=\psi_{\ch_2(E)}(E,V)$ for $V=H^0(Q_4,\cS_+^\vee)\subset H^0(E)$, and therefore $Q\in Y_{\ch_2(E)}=\overline{\Sigma_{\ch_2(E)}}$. Similarly, for $\ch_2(E_-)\geq0$ by dualizing \eqref{ses-pushforward} one can deduce $Q\in\overline{\Sigma}$ as well. 

Finally, the case where $Q\in|I_S(2)|$ has rank $5$ and $\Sing(Q)\cap S\neq\emptyset$ admits a similar treatment: pulling back (to an open subset of $S$) the spinor bundle of a 3-dimensional quadric, one checks that $Q\in\overline{\Sigma}$ as well, which finishes the proof.
\end{proof}

\begin{rem}\leavevmode
\begin{enumerate}
\item From the proof, we can see that the maps $\psi_\ell$ are birational for $\ell\ge1$, while $\psi_0$ is generically two-to-one. In particular, the locus of quadrics of rank $\le 6$ is pure-dimensional of the expected dimension $2g-8$.
\item As an example, when $g=6$, recall that a general $K3$ surface of genus $6$ is a $(1,1,1,2)$-complete intersection in $\Gr(2,5)$.
$\overline{\Sigma}\subset |I_S(2)|$ is a hypersurface defined by the discriminant which has degree $7$.
In the decomposition $\overline\Sigma=Y_0\cup Y_1$, $Y_0$ is a EPW sextic while $Y_1$ is the hyperplane parametrizing Plücker quadrics
(see \cite[Section~5.4]{ogrady} for details).
\item There exists a formula for the degree of $\overline\Sigma$ \cite[Proposition~12]{harris-tu}.
Therefore, the decomposition $\overline\Sigma=\bigcup_{\ell=0}^{\lfloor\frac g2\rfloor-2} Y_\ell$ yields a combinatorial identity
\[
\deg\overline\Sigma
=\prod_{k=0}^{2}\frac{\binom{2g-5}{g-2k}}{\binom{2g-5}{2k}}
= \sum_{\ell=0}^{\lfloor\frac g2\rfloor-2} \deg Y_\ell.
\]
We note that for example $\deg Y_0=\frac{(2g-9)!}{(g-5)!}$,%
\footnote{Using the fact that $Y_0$ admits a (rational) two-to-one covering by the moduli space $\cM(2,L,2)$ via $\psi_0$.}
and when $g$ is even, we have $\deg Y_{\frac g2-2}=\deg\Gr(4, \frac g2+2)=\frac{12\cdot(2g-8)!}{\left(\frac g2-2\right)!\left(\frac g2-1\right)!\left(\frac g2\right)!\left(\frac g2+1\right)!}$.
\end{enumerate}

\end{rem}

\section{Proof of \autoref{thm:L-2delta-very-ample}}
\label{sec:proof-thm-A}

\subsection{Very ampleness for \texorpdfstring{$g\geq 7$}{g≥7}}
Let $(S,L)$ be a polarized $K3$ surface of genus $g$, with $\Pic(S)=\bZ\cdot L$. Given a stable vector bundle $E$ with $v(E)=(2,L,\lfloor\frac{g}{2}\rfloor)$, we consider the rational map
\[
\varphi_E\colon S^{[2]}\dasharrow \Gr\left(\lfloor\tfrac{g}{2}\rfloor-2, H^0(E)\right)\into \bP\left(\bw{\lfloor\frac{g}{2}\rfloor-2}H^0(E)\right)
\]
sending a general $\xi\in S^{[2]}$ to the codimension $4$ subspace $H^0(E\otimes\cI_\xi)\subset H^0(E)$. The map $S^{[2]}\dasharrow \Gr\left(\lfloor\frac{g}{2}\rfloor-2, H^0(E)\right)$ is naturally obtained from the evaluation map
\[
H^0(E)\otimes\cO_{S^{[2]}}\to E^{[2]},
\]
where $E^{[2]}$ is the tautological rank $4$ bundle on $S^{[2]}$ induced by $E$. In particular,
\[
\varphi_E^*\cO(1)=\det(E^{[2]})=L_2-2\delta.
\]

Note that the vector bundle $E$ is unique when $g$ is even. In that case, the following result is enough to conclude the very ampleness of the complete linear system $|L_2-2\delta|$:

\begin{prop}\label{veryampleeven}
    If $g\geq8$ is even, then $\varphi_E$ is a closed immersion.
\end{prop}
\begin{proof}
    First we check that the rational map is actually a morphism. We have to show that for every $\xi\in S^{[2]}$ the equality $h^0(E\otimes\cI_\xi)=\frac{g}{2}-2$ holds, namely $h^1(E\otimes\cI_\xi)=0$. If $0\neq h^1(E\otimes\cI_\xi)=\ext^1(E^\vee,\cI_\xi)$, then by Serre duality we have a nontrivial extension
    \[
    0\to E^\vee \to T\to \cI_\xi\to 0
    \]
    which is easily seen to be slope stable. But then $-2\leq v(T)^2=-g+4$, which is a contradiction. This proves that $\varphi_E$ is regular.

    Now we want to check that $\varphi_E$ defines a closed immersion. To this end, we first observe that $h^1(E\otimes\cI_\eta)\leq 1$ for any $\eta\in S^{[3]}$; otherwise, there is a nontrivial extension
    \[
    0\to (E^\vee)^{\oplus 2}\to T'\to \cI_\eta\to 0
    \]
    which must be slope stable. Then $-2\leq v(T')^2=-2g+12$, leading again to a contradiction.   To conclude we need to show that for any subscheme $\Xi\subset S^{[2]}$ of length $2$, the evaluation map
    \begin{equation}
    \label{eq:E2-evaluation-map}
        H^0(E^{[2]})\to H^0(E^{[2]}|_\Xi),
    \end{equation}
    as a map between vector spaces, has rank $\ge 5$. Indeed, if this is the case, then the kernel has dimension $\le \frac g2 -3$, so the subscheme $\Xi$ cannot be contracted by $\varphi_E$. Note that by definition $E^{[2]}\coloneqq\pi_*\psi^*E$,where $\pi$ and $\psi$ are as in diagram \eqref{eq:universallenght2subscheme}, so the evaluation map in \eqref{eq:E2-evaluation-map} is the same as
\[
H^0(\psi^*E)\to H^0(\psi^*E|_{\pi^{-1}(\Xi)}).
\]
Since $\psi_*\cO_B\cong \cO_S$, we can push forward to $S$ via $\psi$ and get a factorization
\[
H^0(E)\to H^0(E|_{\psi(\pi^{-1}(\Xi))})\stackrel{\psi^*}\into H^0(\psi^*E|_{\pi^{-1}(\Xi)}).
\]
Note that we can always choose a subscheme $\eta\subset\psi(\pi^{-1}(\Xi))$ of length $3$.
Then the map $H^0(E)\to H^0(E|_\eta)$ already has rank $\ge 5$ (as follows from the estimate $h^1(E\otimes\cI_\eta)\leq 1$ above), so the map $H^0(E)\to H^0(E|_{\psi(\pi^{-1}(\Xi))})$ also has rank $\geq5$ and we can conclude.
\end{proof}

\begin{rem} \label{sepeven} \leavevmode
In the language of \autoref{sec:prelim},  \autoref{veryampleeven} can be rephrased as saying that quadrics in $Y_{\frac{g}{2}-2}$ are enough to embed $S^{[2]}$. For example, for any two distinct $\xi_1,\xi_2\in S^{[2]}$ one can find $V\in\Gr(4,H^0(E))$ such that $V\cap H^0(E\otimes\cI_{\xi_1})\neq0$ and $V\cap H^0(E\otimes\cI_{\xi_2})=0$. 
In view of \autoref{containmentcrit}, the corresponding quadric $Q_{E,V}$ contains the line $\langle\xi_1\rangle$, but not $\langle\xi_2\rangle$.
\end{rem}

If $g$ is odd, then the moduli space $\cM\coloneqq\cM(2,L,\lfloor\frac{g}{2}\rfloor)$ is 2-dimensional. In that case we can combine all the rational maps $\varphi_E$ (with $E$ varying in $\cM$) to prove:

\begin{prop}\label{veryampleodd}
    If $g\geq 7$ is odd, then the line bundle $L_2-2\delta$ is very ample on $S^{[2]}$. 
\end{prop}
\begin{proof}
    We already know that the line bundle $L_2-2\delta$ is globally generated, since the $K3$ surface $S\subset \bP(H^0(L)^\vee)$ is cut out by quadrics. In order to check that $L_2-2\delta$ is very ample it suffices to show that, for any subscheme $\Xi\subset S^{[2]}$ of length $2$,
the following properties hold true for a \emph{general} $E\in\cM$:
\begin{enumerate}
    \item\label{oddveryamp-1} $H^1(E\otimes\cI_\xi)=0$ for any closed point $\xi\in S^{[2]}$ in the support of $\Xi$. 
    
    \item\label{oddveryamp-2} The evaluation map 
$H^0(E^{[2]})\to H^0(E^{[2]}|_\Xi)$ has rank $\ge 5$.
\end{enumerate}

Indeed, if this is the case, the indeterminacy locus of $\varphi_E\colon S^{[2]}\dasharrow  \bP\left(\bw{\lfloor\frac{g}{2}\rfloor-2}H^0(E)\right)$ avoids $\Xi$ by \eqref{oddveryamp-1}, and the subscheme $\Xi$ cannot be contracted by $\varphi_E$ thanks to \eqref{oddveryamp-2}. Since
$\varphi_E^*\cO(1)=L_2-2\delta$, it follows that the complete linear system $|L_2-2\delta|$ also separates $\Xi$.
    
In order to prove condition \eqref{oddveryamp-1}, observe that if $h^1(E\otimes\cI_\xi)\neq0$ for some $\xi\in S^{[2]}$ and $E\in\cM$, we have a slope stable nontrivial extension
    \[
    0\to E^\vee \to T\to \cI_\xi\to 0.
    \]
    Then $-2\leq v(T)^2=-g+7$, which gives a contradiction if $g\geq 11$. For the remaining cases, note that for a fixed $\xi\in S^{[2]}$ the condition $h^1(E\otimes\cI_\xi)>0$ holds for at most two vector bundles $E\in\cM$ if $g=7$ (see for instance \cite[Section~3]{moretti-rojas}), whereas if $g=9$ it holds for at most one bundle $E\in\cM$ (this is Proposition~4.1 in \emph{loc. cit.}).

    On the other hand, arguing as in the proof of \autoref{veryampleeven}, we prove \eqref{oddveryamp-2} by verifying that for any fixed $\eta\in S^{[3]}$, the general element $E\in\cM$ satisfies $h^1(E\otimes\cI_\eta)\leq 1$. If $h^1(E\otimes\cI_\eta)\geq 2$ for some $\eta\in S^{[3]}$ and $E\in\cM$, we have a slope stable nontrivial extension
    \[
    0\to (E^\vee)^{\oplus 2} \to T'\to \cI_\eta\to 0
    \]
    and hence $-2\leq v(T')^2=-2g+22$, which gives a contradiction for $g\geq 13$. The remaining cases can again be argued separately. Fix $\eta\in S^{[3]}$. The condition $h^1(E\otimes\cI_\eta)\geq 2$ holds for at most one bundle $E\in\cM$ if $g=11$ (\cite[Section~5]{moretti-rojas}), and for at most three bundles $E\in\cM$ if $g=9$ (this follows from Lemma~2.13 in \emph{loc. cit.}) or $g=7$ (this is a consequence of the isomorphism $S^{[3]}
    \cong\cM^{[3]}$ described in \cite[Theorem~7.6]{yoshioka} via Fourier--Mukai theory).
\end{proof}

\begin{rem}
    In analogy with the case of even genus, the proof of \autoref{veryampleodd} shows that the linear system defined by the linear span of $Y_{\lfloor\frac{g}{2}\rfloor-2}$ in $|I_S(2)|$ induces a closed immersion.
\end{rem}

\subsection{Strange duality} In this subsection we show Le Potier's strange duality conjecture for the pair $(S^{[2]},L_2-2\delta)$ for any $g\geq6$. This generalizes a result by O'Grady \cite[Section~5.3]{ogrady}, who established the conjecture in the range $6\leq g\leq 8$.

Strange duality is a statement about duality of sections of determinant line bundles on moduli spaces of sheaves. Whereas we refer the reader to \cite{marian-oprea-tour} for a gentle introduction to the subject and further details, here we restrict ourselves to the Mukai vectors $v\coloneqq(1,0,-1)$ and $w\coloneqq(2,L,2)$ on $S$. 

Note that $v^2=2$ and $w^2=2g-10\geq2$. According to \eqref{thetaiso}, we have isometries
\[
\theta_v\colon v^\perp\overset{\cong}{\to} \Pic(\cM(v))=\Pic(S^{[2]}),\quad \theta_w\colon w^\perp\overset{\cong}{\to} \Pic(\cM(w))
\]
defined via determinant line bundles.
In particular, since $\langle v,w\rangle =0$, we have an induced line bundle $H\coloneqq\theta_w(v)$ on $\cM(w)$. Furthermore, by virtue of \autoref{explicit-thetaiso} the equality $L_2-2\delta=\theta_v(w)$ holds.

We have $h^2(E\otimes\cI_\xi)=0$ for every $E\in\cM(w)$ and $\xi\in S^{[2]}$, and the jump locus
\[
\Theta\coloneqq\bigl\{(E,\xi)\in\cM(w)\times S^{[2]}: h^1(E\otimes\cI_\xi)>0\bigr\}
\]
defines a divisor in the product $\cM(w)\times S^{[2]}$, corresponding to the line bundle $H\boxtimes (L_2-2\delta)$. Therefore, we have a natural (up to constant) morphism
\begin{equation}\label{SD}
H^0(S^{[2]},L_2-2\delta)^\vee\to H^0(\cM(w),H).
\end{equation}
The strange duality conjecture predicts the following:

\begin{conj}[Strange duality]\label{SD-statement}
For any $g\geq6$, the map \eqref{SD} is an isomorphism.
\end{conj}

In the range $6\leq g \leq 8$, O'Grady \cite{ogrady} proved strange duality employing the following arguments:

\begin{enumerate}
    \item\label{SD-cond1} We have $h^0(S^{[2]}, L_2-2\delta)=\binom{g-2}{2}$ (as $S\subset \bP^g$ is projectively normal) and $h^0(\cM(w),H)=\chi(\cM(w),H)=\binom{g-2}{2}$ (as $H$ is ample for $g\leq 8$).
    
    \item\label{SD-cond2} Let $f\colon\cM(2,L,2)\dasharrow |I_S(2)|=\bP(H^0(L_2-2\delta))$ be the rational map induced by $\psi_0$ in \eqref{eq:def-relGrass}. Then $f$ is a morphism, and $f^*\cO(1)=H$.

    \item\label{SD-cond3} $Y_0=\im(f)\subset |I_S(2)|$ is a non-degenerate subvariety. Therefore, $f^*$ induces an injection
    \[
    |I_S(2)|^\vee\into |H|
    \]
    which must be an isomorphism, by the equality of dimensions in \eqref{SD-cond1}.

    \item The pullback map $f^*\colon |I_S(2)|^\vee\to|H|$ is  the projectivization of \eqref{SD}.
\end{enumerate}

In the case of an arbitrary $g\geq 6$, we observe that:

\begin{enumerate}
    \item[(1')] The equalities $h^0(S^{[2]}, L_2-2\delta)=\binom{g-2}{2}=h^0(\cM(w),H)$ still hold, as $H$ is big and nef for $g\geq 9$ (see for instance \cite[Example B.3]{debarre}).
    
    \item[(2')] For $g\geq9$ the map $f$ is not regular, with indeterminacy given by the Brill--Noether locus
    \[
    \{E\in\cM(w):\;h^1(E)>0\}
    \]
    which has codimension 5 in $\cM(w)$ (see \cite{markman} for a more detailed description). The argument in \cite[Lemma 5.9]{ogrady} still shows $f^*\cO(1)=H$ in this case.

    \item[(4')] The induced map $f^*\colon |I_S(2)|^\vee\to|H|$ is still the projectivization of \eqref{SD}. This can be argued exactly as in \cite[Section 5.3]{ogrady}, by using \autoref{containmentcrit} (which generalizes Claim~5.17 in \emph{loc.~cit.}).
\end{enumerate}

Therefore, in order to establish \autoref{thm:L-2delta-very-ample}.\eqref{thm:L-2delta-very-ample-2} it suffices to prove \eqref{SD-cond3} for any $g\geq6$, namely that $Y_0\subset |I_S(2)|$ is non-degenerate. Note that the argument in \cite[Lemma 5.11]{ogrady} cannot be applied, as it relies on an explicit description of $|I_S(2)|$ in terms of the Mukai model of $S$.

To circumvent this, we consider a general curve $C\in|L|$, which is Brill--Noether--Petri general by \cite{laz}. Consider the cartesian diagram
\[
\begin{tikzcd}
C \arrow[hookrightarrow,"i"]{d}\arrow[hookrightarrow]{r} & \bP^{g-1}=\bP(H^0(\omega_C)^\vee) \arrow[hookrightarrow]{d} \\
S \arrow[hookrightarrow]{r}                          & \bP^{g}=\bP(H^0(L)^\vee) \
\end{tikzcd}
\]
inducing a natural identification of quadrics $|I_{C/\bP^{g-1}}(2)|=|I_{S}(2)|$. 
According to a theorem of Green \cite{green}, $|I_{C/\bP^{g-1}}(2)|$ is spanned by a ($g-4$)-dimensional subvariety $\mathcal{Z}$ parametrizing certain quadrics of rank $\leq 4$. $\mathcal{Z}$ is birational to the quotient $W^1_{g-1}(C)/\langle \iota\rangle$, where $\iota$ is the involution given by $\iota(A)=\omega_C\otimes A^{-1}$. 

The general element of $\mathcal{Z}$ is given as follows. Fix $A\in W^1_{g-1}(C)$ with $h^0(A)=2$, such that both $A$ and $\omega_C\otimes A^{-1}$ are basepoint-free. Then the Segre quadric (parametrizing decomposable tensors in $\bP\left(H^0(A)^\vee\otimes H^0(\omega_C\otimes A^{-1})^\vee\right)$) induces via the Petri map
\[
m_A:H^0(A)\otimes H^0(\omega_C\otimes A^{-1})\hookrightarrow H^0(\omega_C)
\]
a quadric $Q_A$ of rank $\leq 4$ in $\bP(H^0(\omega_C)^\vee)$ containing the canonical curve $C$.
In coordinates: if $s_1,s_2$ and $t_1,t_2$ denote respective bases of $H^0(A)$ and $H^0(\omega_C\otimes A^{-1})$, then
\[
Q_A=\det\begin{pmatrix}
m_A(s_1\otimes t_1) & m_A(s_1\otimes t_2)  \\ 
m_A(s_2\otimes t_1) & m_A(s_2\otimes t_2) 
\end{pmatrix}\in \Sym^2 H^0(C,\omega_C).
\]

In view of Green's theorem, non-degeneracy of $Y_0$ in $|I_S(2)|$ is implied by the following:

\begin{prop}
    Under the natural identification $|I_{C/\bP^{g-1}}(2)|=|I_{S}(2)|$, there is a closed immersion $\mathcal{Z}\subset Y_0$.
\end{prop}
\begin{proof}
    We consider the general element of $\mathcal{Z}$, given by $A\in W^1_{g-1}(C)$ as before. Consider the associated Lazarsfeld--Mukai bundle $E$ of $A$ (c.f. \cite{laz,aprodu}) fitting in
    \begin{equation}\label{LMbundle}
        0\to H^0(A)^\vee\otimes\cO_S\to E\to i_*(\omega_C\otimes A^{-1})\to 0.
    \end{equation}
    We have $E\in\cM(2,L,2)$ (indeed, $c_2(E)=\deg(A)=g-1$) is a globally generated vector bundle with $h^0(E)=4$. Moreover:

\begin{enumerate}
    \item[(a)]\label{LM-A} $\bigwedge^2 H^0(A)^\vee$ maps, under the determinant map $\bigwedge^2 H^0(E)\to H^0(L)$, to a section in $H^0(L)$ whose zero locus is $C$.

    \item[(b)]\label{LM-B} Under the natural identification $\bP(H^0(A)^\vee)\cong \bP(H^0(A))$, the zero locus of $s\in H^0(A)^\vee$ (regarded as a global section of $E$ via \eqref{LMbundle}) lies in $C$ and describes the corresponding divisor in the linear system $|A|$.
\end{enumerate}

Similarly, we consider the Lazarsfeld--Mukai bundle $F$ of $\omega_C\otimes A^{-1}$. One immediately checks that $F\in\cM(2,L,2)$, and $F^\vee$ (resp.~$E^\vee$) is the kernel of the evaluation map of global sections of $E$ (resp.~$F$). In particular, we have an equality of quadrics $Q_{E,H^0(E)}=Q_{F,H^0(F)}\in Y_0$. 
Denoting this quadric by $Q_E$, we will check that its hyperplane section $Q_E\cap \bP(H^0(\omega_C)^\vee)$ equals the rank-4 quadric $Q_A$; this is enough to derive the inclusion $\mathcal{Z}\subset Y_0$. 

To this end, we pick bases $s_1,s_2$ of $H^0(A)$, and $t_1,t_2$ of $H^0(\omega_C\otimes A^{-1})$. Moreover, we pick a basis $x_0,x_1,...,x_g$ of $H^0(L)$, in such a way that  $C=S\cap\{x_0=0\}$ and $x_1=m_A(s_1\otimes t_1)$, $x_2=m_A(s_1\otimes t_2)$, $x_3=m_A(s_2\otimes t_1)$, $x_4=m_A(s_2\otimes t_2)$.
 In particular, for the induced basis $x_1,...,x_g$ of $H^0(\omega_C)$, we have that $x_1x_4-x_2x_3\in \Sym^2 H^0(\omega_C)$ is an equation for $Q_A$.

On the other hand, under the exact sequence $0\to H^0(A)^\vee\to H^0(E)\to H^0(\omega_C\otimes A^{-1})\to 0$, we consider the basis $s_1^*,s_2^*,\tilde{t_1},\tilde{t_2}$ of $H^0(E)$, where:
        \begin{itemize}
            \item $s_1^*,s_2^*$ is the basis dual to $s_1,s_2$.
            \item $\tilde{t_i}\in H^0(E)$ is a lift of $t_i\in H^0(\omega_C\otimes A^{-1})$.
        \end{itemize}
We also consider the induced basis $y_0=s_1^*\wedge s_2^*$, $y_1=s_1^*\wedge \tilde{t_1}$, $y_2=s_1^*\wedge \tilde{t_2}$, $y_3=s_2^*\wedge \tilde{t_1}$, $y_4=s_2^*\wedge \tilde{t_2}$, $y_5=\tilde{t_1}\wedge \tilde{t_2}$ of $\bw2H^0(E)$. The natural map
        \begin{equation}\label{natmap-quadric}
            \bw2H^0(E)\to H^0(L)\to H^0(\omega_C)
        \end{equation}
sends $s_1^*\wedge s_2^*$ to $0$ by (a), and $s_i^*\wedge\tilde{t_j}$ to $m_A(s_i\otimes t_j)$ by (b). It follows that the quadric $Q_E\cap \bP(H^0(\omega_C)^\vee)\in\Sym^2 H^0(\omega_C)$, which is obtained from the Plücker quadric $y_0y_5-y_1y_4+y_2y_3\in\Sym^2 \left(\bigwedge^2 H^0(E)\right)$ via \eqref{natmap-quadric}, is given by the equation $-x_1x_4+x_2x_3\in\Sym^2 H^0(\omega_C)$. The proof is complete, as this equation is proportional to the one of $Q_A$.
\end{proof}

\section{Genus 7}
\label{sec:genus-7}

\subsection{Quadrics containing the Mukai model}

Let $(S,L)$ be a polarized $K3$ surface of genus~$7$ with $\Pic(S)=\bZ\cdot L$. We first give a brief description of the Mukai model for the embedding $S\into \bP(H^0(L)^\vee)=\bP^7$, following \cite{mukai-models}.

Let $E$ be the unique stable vector bundle with Mukai vector $v(E)=(5,2L,5)$. The evaluation map of global sections fits in a short exact sequence 
\begin{equation}\label{globgeneration7}
    0\to E^\vee \to H^0(E)\otimes\cO_S\to E\to 0,
\end{equation}
and defines an immersion $S\into \Gr(5, H^0(E)^\vee)$ by the rule $p\mapsto H^0(E\otimes\cI_p)^\perp$.

By \cite[Proposition 3.4]{mukai-models}, the 10-dimensional vector space $H^0(E)^\vee$ is equipped with a canonical (up to scalar) non-degenerate symmetric bilinear form $q$, corresponding to the kernel of the natural map
\[
\Sym^2H^0(E)\to H^0(\Sym^2E).
\]
By construction, $H^0(E\otimes\cI_p)^\perp\subset H^0(E)^\vee$ is an isotropic subspace with respect to $q$ for all $p\in S$.
In this way, $S$ is mapped into $\OGr(5,H^0(E)^\vee)_+$, one of the two connected components of the $10$-dimensional orthogonal Grassmannian parametrizing isotropic subspaces with respect to $q$.

The Plücker line bundle on $\OGr(5,H^0(E)^\vee)_+$ coming from $\OGr(5,H^0(E)^\vee)_+\into \bP(\bw5H^0(E)^\vee)$ is not primitive: it equals twice a very ample line bundle $\cO_{\OGr_+}(1)$ inducing the \emph{spinor embedding}
\[
\OGr(5,H^0(E)^\vee)_+\into \bP^{15}=\bP V_+.
\]
The $K3$ surface $S\into \bP(H^0(L)^\vee)=\bP^7$ is then a $\bP^7$-linear section of $\OGr(5,H^0(E)^\vee)_+$.

For later use, let us give some details on the spinor embedding of $\OGr(5,H^0(E)^\vee)_+$. We follow the notation of \cite[Section 1]{mukai-genus7}. Say $V=H^0(E)^\vee$, and pick  two disjoint isotropic subspaces $U_0,U_\infty\subset V$, so that $V\cong U_0\oplus U_\infty$ and $U_0\cong U_\infty^\vee$. There is a linear map
\[
V\to \mathrm{End}(\bw\bullet U_\infty),\quad v\mapsto \varphi_v
\]
where $\varphi_v$ is defined by wedge (resp.~by derivation) if $v\in U_\infty$ (resp.~if $v\in U_0=U_\infty^\vee$). If we are taking $U_0\in\OGr(5,H^0(E)^\vee)_+$, then the spinor embedding
\[
\OGr(5,H^0(E)^\vee)_+\into \bP^{15}=\bP\left(\bw{\mathrm{even}} U_\infty\right),\quad U\mapsto [s_U]
\]
sends an isotropic subspace $U$ to the unique (up to constant) spinor $s_U$ with the property $\varphi_u(s_U)=0$ for every $u\in U$. For instance, $s_{U_0}=1\in \bw0 U_\infty$. 
Similarly, the other component of the orthogonal Grassmannian admits an embedding
\[
\OGr(5,H^0(E)^\vee)_-\into (\bP^{15})^\vee=\bP\left(\bw{\mathrm{odd}} U_\infty\right).
\]

Regarding $S$ as a codimension $8$ linear section of $\OGr(5,H^0(E)^\vee)_+$, there is an identification of $|I_{S}(2)|$ with $\bP H^0\left(\cI_{\OGr_+/\bP^{15}}(2)\right)$; the latter can be identified with $\bP(H^0(E)^\vee)$ as follows.
Let $\kappa\colon\bw\bullet U_\infty\to\bw5 U_\infty\cong \bC$ denote the projection onto the top part, and consider the bilinear form $\beta$ on $\bw\bullet U_\infty$ defined by the rule 
\[
\beta(\xi,\xi')=(-1)^{\frac{p(p+1)}{2}}\kappa(\xi\wedge\xi'),
\]
for $\xi$ homogeneous of degree $p$. Then for every $v\in V$, the rule $N_v(s)=\beta(s,\varphi_v(s))$ defines a quadratic form $N_v$ on $\bw{\mathrm{even}} U_\infty$, such that $H^0\left(\cI_{\OGr_+/\bP^{15}}(2)\right)=\{N_v\}_{v\in V}$.

\begin{rem}\label{bilformgenus7}
One can check that $\beta(s_1, \varphi_v(s_2))=\beta(\varphi_v(s_1), s_2)$ for all $s_1,s_2\in \bw{\mathrm{even}} U_\infty$ and $v\in V$. This describes the symmetric bilinear form whose associated quadratic form is $N_v$.
\end{rem}

\subsection{The Fourier--Mukai partner}\label{subsec:FMpartner} The geometry of the Mukai model for $S$ is very much related to the Mukai model for the (fine) moduli space $\cM\coloneqq\cM(2,L,3)$, which is another $K3$ surface whose Picard group is generated by a polarization $\hat{L}$ of genus 7. Let us give a brief account; the reader is referred to \cite{iliev-markushevich,kuznetsov} for details.

Let $\cF$ denote a universal rank $2$ bundle on $S\times\cM$; for an appropriate choice (obtained by twisting $\cF$ with the pullback of a line bundle on $\cM$) $\cF$ also becomes the universal family of vector bundles on $\cM$ with Mukai vector $(2,\hat{L},3)$. Given a point $p\in S$, we will denote by $\cF_p$ the corresponding vector bundle on $\cM$.

Denote by $p_S$ and $p_\cM$ the respective projections of $S\times \cM$ onto $S$ and $\cM$. Then the rank $5$ bundle $E$ on $S$ can be recovered as $p_{S*}\cF$, in particular $H^0(E)=H^0(\cF)$. By symmetry $E_\cM\coloneqq p_{\cM*}\cF$ is the stable vector bundle with Mukai vector $(5,2\hat{L},5)$, and the immersion $\cM\into \Gr(5, H^0(E)^\vee)$ factors through the other component $\OGr(5,H^0(E)^\vee)_-$ of the orthogonal Grassmannian.
If we consider the spinor embedding of $\OGr(5,H^0(E)^\vee)_-$ in the dual space $(\bP^{15})^\vee$, then $\cM$ becomes the linear section of $\OGr(5,H^0(E)^\vee)_-$ by the 7-dimensional linear space which is orthogonal to the linear span of $S$ in $\bP^{15}$. 

It is also worth mentioning that $(S^{[2]},L_2-2\delta)\cong(\cM^{[2]},\hat{L}_2-2\hat\delta)$ as polarized varieties. Indeed the isomorphism $S^{[2]}\cong\cM^{[2]}$ can be realized by means of a Fourier--Mukai transform (see \cite[Theorem 7.6]{yoshioka}), and a simple numerical computation shows that the polarization is preserved. Whereas we will provide another geometric interpretation for this isomorphism in \autoref{rem:S2-iso-fm-partner}, let us give some more details on its Fourier--Mukai realization. We denote by
\[
\Phi\coloneqq\Phi_\cF^{S\to\cM}\colon\Db(S)\to\Db(\cM),\quad\hat{\Phi}\coloneqq\Phi_\cF^{\cM\to S}\colon\Db(\cM)\to\Db(S)
\]
the Fourier--Mukai equivalences with kernel $\cF$. The isomorphism $\iota\colon\cM^{[2]}\to S^{[2]}$ is constructed in \cite{yoshioka} by proving that, for a given $\tau\in\cM^{[2]}$, the derived dual $\hat{\Phi}(\cI_\tau)^\vee$ is the ideal sheaf of a length-2 subscheme $\xi$ on $S$; one sets $\iota(\tau)=\xi$.
When $\tau=\{F_1,F_2\}$ is reduced, the short exact sequence
\[
0\to \cI_\tau\to \cO_\cM\to \cO_\tau\to 0
\]
on $\cM$ induces a short exact sequence
\[
0\to \hat{\Phi}(\cO_\tau)^\vee=F_1^\vee\oplus F_2^\vee\to \hat{\Phi}(\cO_\cM)^\vee=E^\vee\to \hat{\Phi}(\cI_\tau)^\vee=\cI_{\iota(\tau)}\to 0
\]
on $S$ (similarly, if $\tau$ is non-reduced and supported at $F\in\cM$, then $\hat{\Phi}(\cO_\tau)^\vee$ is a self-extension of $F^\vee$). In this way, for a given $\xi\in S^{[2]}$ the support of $\iota^{-1}(\xi)$ consists of the vector bundles $F\in\cM$ with $h^1(S,F\otimes\cI_\xi)=1$ (or equivalently, $h^0(S,F\otimes\cI_\xi)=2$); see also \cite[Section~3]{moretti-rojas} for further details.

\begin{rem}\label{rem-FMkernelbundle}
    There is a well-known isomorphism $\cM\overset{\cong}{\to}\cM(3,-L,2)$ mapping any $F\in\cM$ to the unique vector bundle $K_F\in\cM(3,-L,2)$ sitting in an extension
    \[
    0\to F^\vee=\hat{\Phi}(\cO_{[F]})^\vee\to E^\vee=\hat{\Phi}(\cO_\cM)^\vee\to K_F\to 0
    \]
    (see \cite[Lemma~3.1]{moretti-rojas}), hence $K_F=\hat{\Phi}(\cI_{[F]})^\vee$. Moreover, it is easy to check that any $F\in\cM$ is globally generated, and the kernel bundle $\ker(\ev)$ sitting in the short exact sequence
    \[
    0\to \ker(\ev)\to H^0(S,F)\otimes\cO_S\overset{\ev}{\to} F\to 0
    \]
    is stable of Mukai vector $(3,-L,2)$. In particular, since every morphism $\cM\to\cM$ is either constant or the identity, it follows that $\ker(\ev)=K_F$.
\end{rem}

\subsection{Geometry of the linear system \texorpdfstring{$|L_2-2\delta|$}{|L-2δ|}}
Given two distinct points $p,q\in S$, following \cite[Proposition 1.6]{mukai-genus7} the intersection
\[
H^0(E\otimes\cI_p)^\perp\cap H^0(E\otimes\cI_q)^\perp=\left(H^0(E\otimes\cI_p)+H^0(E\otimes\cI_q)\right)^\perp
\]
is odd-dimensional, hence so is the subspace $H^0(E\otimes\cI_{p,q})=H^0(E\otimes\cI_p)\cap H^0(E\otimes\cI_q)$. In particular, $h^0(E\otimes\cI_{p,q})\neq0$. More concretely, we have:

\begin{lem}\label{sectionsE}
The following hold:
\begin{enumerate}
    \item\label{item1:sectionsE} $h^0(E\otimes\cI_\xi)=1$ for every $\xi\in S^{[2]}$.
    \item\label{item2:sectionsE} Given $s\in H^0(E)$, the zero locus $Z(s)$ is either empty or $0$-dimensional of length $\leq 2$.
\end{enumerate}
\end{lem}
\begin{proof}
By the above analysis and semicontinuity, we already know that $h^0(E\otimes\cI_\xi)\geq1$ for every $\xi\in S^{[2]}$. The equality follows from the uniqueness of the Harder--Narasimhan filtration of $\cI_\xi$ with respect to certain Bridgeland stability conditions (see \cite[Section 3.3]{moretti-rojas} for details). This proves \eqref{item1:sectionsE}.

In order to prove \eqref{item2:sectionsE}, first observe that if non-empty, $Z(s)$ must be 0-dimensional; otherwise, the short exact sequence
\[
0\to K\to E^\vee \overset{s}{\to} \cI_{Z(s)}\to 0
\]
would contradict the stability of $E^\vee$. Now assume for the sake of a contradiction that $Z(s)$ has length $m\geq 3$.
We first show that $K$ is slope stable.

Indeed, $K$ is clearly slope semistable (as $E^\vee$ is stable). If it were strictly semistable, then it would be an extension of two slope stable sheaves $K_1,K_2$ of rank 2 and $c_1=-L$; it is easy to see that $m\geq 3$ implies either $v(K_1)^2<-2$ or $v(K_2)^2<-2$, which is impossible.

Therefore, $K$ is slope stable. But then $m\geq 3$ implies $v(K)^2<-2$, contradiction.
\end{proof}

An immediate consequence of \autoref{sectionsE}.\eqref{item1:sectionsE} is the existence of a natural morphism
\[
\phi\colon S^{[2]}\to \bP(H^0(E)),\quad \xi\mapsto H^0(E\otimes\cI_\xi).
\]

\begin{prop}\label{mapgenus7}
    Under the identification of $\bP(H^0(E)^\vee)$ with $|I_S(2)|$, the map $\phi$ is defined by the complete linear system $|L_2-2\delta|$ on $S^{[2]}$.
\end{prop}
\begin{proof}
    Given two distinct subspaces $U_1,U_2\in\OGr(5,H^0(E)^\vee)_+\subset\bP^{15}$, we are going to prove that all quadrics in the subspace $U_1+U_2\subset H^0(E)^\vee=H^0\left(\cI_{\OGr_+/\bP^{15}}(2)\right)$ contain the line $\langle U_1,U_2\rangle\subset \bP^{15}$. This will prove that $\phi$ coincides with the map defined by $|L_2-2\delta|$ on the open subset of reduced subschemes (simply take $U_1=H^0(E\otimes\cI_p)^\perp,U_2=H^0(E\otimes\cI_q)^\perp$ for two distinct points $p,q\in S$), hence both maps must be equal.

    So now we consider $U_1,U_2\in\OGr(5,H^0(E)^\vee)_+$ and the corresponding pure spinors $[s_{U_1}],[s_{U_2}]\in \bP(\bw{\mathrm{even}} U_\infty)$. We want to prove that for every $u=u_1+u_2$ ($u_1\in U_1,u_2\in U_2$), the quadratic form $N_u$ vanishes at all spinors of the form
    \[
    \lambda s_{U_1}+\mu s_{U_2}\in \bw{\mathrm{even}} U_\infty,\quad\lambda,\mu\in\bC.
    \]
    Since $N_u$ vanishes at $s_{U_1}$ and $s_{U_2}$, it suffices to check that $N_u$ vanishes at $s_{U_1}+s_{U_2}$. To this end, observe that
    $\varphi_{u_1+u_2}(s_{U_1}+s_{U_2})=\varphi_{u_1}(s_{U_2})+\varphi_{u_2}(s_{U_1})$, and hence
    \begin{align*}
        N_{u_1+u_2}(s_{U_1}+s_{U_2})=\beta(s_{U_1}+s_{U_2},\varphi_{u_1}(s_{U_2})+\varphi_{u_2}(s_{U_1}))= \\
        =\beta(s_{U_1},\varphi_{u_1}(s_{U_2}))+\beta(s_{U_1},\varphi_{u_2}(s_{U_1}))+\beta(s_{U_2},\varphi_{u_1}(s_{U_2}))+\beta(s_{U_2},\varphi_{u_2}(s_{U_1})).
    \end{align*}

    Note that the second and the third summands vanish (they equal $N_{u_2}(s_{U_1})$ and $N_{u_1}(s_{U_2})$, respectively). Moreover, the first summand (and similarly the fourth one) vanishes, as we have $\beta(s_{U_1},\varphi_{u_1}(s_{U_2}))=\beta(\varphi_{u_1}(s_{U_1}),s_{U_2})=0$ thanks to \autoref{bilformgenus7}.
    \end{proof}

In particular, $\phi$ is a closed immersion by \autoref{thm:L-2delta-very-ample}.\eqref{thm:L-2delta-very-ample-1}. Recalling that there is an isomorphism $H^0(E)\cong H^0(E)^\vee$ provided by the bilinear form $q$, and hence $|I_S(2)|\cong |I_S(2)|^\vee$ canonically, we have a more geometric realization of the embedding $\phi$:

\begin{cor}\label{sing-secantline}
    The embedding $S^{[2]}\into|I_S(2)|$ defined by $L_2-2\delta$ sends a length-$2$ subscheme $\xi$ to the unique quadric $Q_\xi\in|I_S(2)|$ whose singular locus contains the line $\langle\xi\rangle$.
\end{cor}
\begin{proof}
    There is a well-known isomorphism $E\cong\cN_{S/\bP^7}^\vee(2)$ (see e.g.~\cite[Corollary~4.6]{kuznetsov-2}), so that the natural identification $I_S(2)=H^0(E)^\vee\cong H^0(E)$ corresponds to the natural map \eqref{quadrics-conormal}.
    Then the assertion immediately follows from \autoref{mapgenus7}.
\end{proof}

\begin{rem}
It follows from \autoref{sing-secantline} that the image of $S^{[2]}$ in $|I_S(2)|$ lies in the locus of quadrics of rank $\leq 6$. More precisely, it is contained in the component $Y_1$ of this locus. Indeed, for distinct points $p,q\in S$, the image of $\{p,q\}$ under the isomorphism $\iota^{-1}:S^{[2]}\to\cM^{[2]}$ is supported at those $F_1,F_2\in\cM$ for which $h^0(F_i\otimes \cI_{p,q})=2$. Following the notation of \eqref{eq:def-relGrass} we take the pair $(F_i,V)\in\cG_1$, where $V=H^0(F_i\otimes\cI_p)+H^0(F_i\otimes\cI_q)$; the image of $\bw2 V\to H^0(L)$ is contained in $H^0(L\otimes\cI_{p,q})$, hence $\psi_1(F_i,V)\in Y_1$ is the quadric singular along the line $pq$.

\end{rem}

By \autoref{sectionsE}, we have a stratification
\[
S^{[2]}\into \Gamma\into\bP(H^0(E))
\]
of the space of global sections according to the length of their zero locus (i.e. $\Gamma$ parametrizes global sections with non-empty zero locus). This stratification fits in a commutative diagram
\begin{equation}
\label{eq:diagram-genus7}
\begin{tikzcd}
B \ar[d,"\pi"]\ar[r,hookrightarrow] & \bP_S(E^\vee) \ar[d]\ar[dr]   \\
S^{[2]}\ar[r,hookrightarrow]  & \Gamma\ar[r,hookrightarrow]  & \bP(H^0(E))
\end{tikzcd}
\end{equation}
where:
\begin{itemize}
    \item $\bP_S(E^\vee)\to\bP(H^0(E))$ is the natural morphism defined by \eqref{globgeneration7}. Its image equals $\Gamma$, as \eqref{globgeneration7} identifies canonically $H^0(E\otimes\cI_p)$ with the fiber of $E^\vee$ at $p\in S$.

    \item The left square is cartesian, namely $B$ is the universal length-$2$ subscheme of \eqref{eq:universallenght2subscheme}.
\end{itemize}

Observe that $\bP_S(E^\vee)\to\bP(H^0(E))$ is birational to its image (defines an isomorphism outside $B$) and $\Gamma$ is non-normal, singular along $S^{[2]}$. Furthermore, there is a unique quadric hypersurface $Q\subset \bP(H^0(E))$ containing $\Gamma$, defined by the bilinear form $q$.%
\footnote{More precisely, the map $E^\vee\to H^0(E)\otimes\cO_S$ in \eqref{bilformgenus7} identifies, via the natural isomorphism $H^0(E)\cong H^0(E)^\vee$, with the map $E^\vee\to H^0(E)^\vee\otimes \cO_S$ obtained by dualizing \eqref{bilformgenus7}. Therefore, quadrics containing $\Gamma$ are parametrized by the kernel of $\Sym^2 H^0(E)\to H^0(\Sym^2 E)$.} 

Note that by using the Riemann--Roch polynomial, one can compute that
\[
h^0(S^{[2]},\cO(2))=55=h^0(\bP^9,\cO(2)).
\]
Therefore, the fact that $S^{[2]}$ lies on the quadric $Q$ implies that $S^{[2]}$ is not projectively normal.

\subsection{Degeneracy locus}
The Hilbert square $S^{[2]}$ can be realized as a degeneracy locus in the smooth quadric $Q\subset \bP(H^0(E))$. This will allow us to describe its ideal in detail in the next subsection.

First let us recall the construction of the spinor bundles on even dimensional quadrics.
We refer to \cite{ottaviani} for details.
Let $Q$ be a smooth quadric hypersurface of dimension $2k$,
and let $\OGr_+\coloneqq\OGr(k+1, 2k+2)_+$ be one of the two corresponding orthogonal Grassmannians,
embedded in $\bP^{2^k-1}=\bP V_+$ via the spinor embedding.
Consider the incidence variety
\begin{equation}
\label{eq:incidence-variety}
\begin{tikzcd}
 &I\ar[ld, "p"']\ar[rd, "q"]\coloneqq\setmid{(x,[V])}{x\in \bP(V)}\\
Q&&\OGr_+
\end{tikzcd}
\end{equation}
Here, the map $q$ is simply the projectivization of the rank-$(k+1)$ universal subbundle on $\OGr_+$.
For the map $p$, each fiber $p^{-1}(x)$ is the subvariety $\setmid{[V]\in \OGr_+}{x\in \bP(V)}$,
which spans a linear subspace of affine dimension $2^{k-1}$.
This globalizes to the spinor bundle $\cS_+$ on $Q$, a vector bundle of rank $2^{k-1}$.
We thus view $\cS_+$ as a natural subbundle of the trivial vector bundle $V_+\otimes\cO_Q$, where $V_+\coloneqq H^0(\OGr_+, \cO(1))^\vee$ is the space of spinors.

Going back to our case where $Q$ is an $8$-dimensional quadric. We note that the fibers $p^{-1}(x)$ are isomorphic to the previous orthogonal Grassmannian $\OGr(4,8)$, which in turn is isomorphic to the $6$-dimensional quadric.
In the above picture of the incidence variety, these $6$-dimensional quadrics each span a $\bP^7$ that globalize to the rank-$8$ spinor bundle $\cS_+$.

On the other hand, recall that the $K3$ surface $S$ is determined by the choice of the $8$-dimensional linear subspace
\[
H^0(L)^\vee\subset V_+.
\]
Equivalently, we get a quotient map $V_+\onto W$ where $W$ is also of dimension $8$.
Via composition, we therefore obtain a morphism between two vector bundles of rank $8$
\[
\varphi\colon \cS_+\into V_+\otimes\cO_Q \onto W\otimes\cO_Q.
\]
\begin{prop}
\label{prop:S2-is-D6}
Scheme-theoretically, $S^{[2]}$ coincides with $D_6(\varphi)$.
\end{prop}
\begin{proof}
By the description of the map $\phi\colon S^{[2]} \to Q \subset \bP^9=\bP(H^0(E))$,
if we take a reduced subscheme $\xi\coloneqq\{x,y\}\in S^{[2]}$,
the corresponding point $\phi(\xi)\in Q$ will be contained in two different isotropic subspaces $V_x=H^0(E\otimes\cI_x)$ and $V_y=H^0(E\otimes \cI_y)$,
whose classes $[V_x]$ and $[V_y]$ both lie on the $K3$ surface $S$.
Therefore, we have
\[
\dim \cS_{+,\xi}\cap H^0(L)^\vee \ge 2,\quad \rk \varphi_\xi \le 6.
\]
and $S^{[2]}$ must be contained in the degeneracy locus $D_6(\varphi)$.

Conversely, if $x\in Q$ satisfies $\rk\varphi_x = 6$,
then the scheme-theoretic intersection $\bP(\cS_{+,x})\cap S$ contains a subscheme of length  $\ge 2$ (as $p^{-1}(x)\subset \bP(\cS_{+,x})$ is a non-degenerate hypersurface); it must be of length exactly 2, since there are no trisecant lines to $S$. From this it follows that $p^{-1}(x)\subset \bP(\cS_{+,x})$ has degree 2. If $\rk\varphi_x \leq 5$, then $\bP(\cS_{+,x})\cap S$ is at least one-dimensional and contains a conic, which is impossible under the assumption $\Pic(S)=\bZ\cdot L$. Thus $S^{[2]}$ and $D_6(\varphi)$ are equal set-theoretically.

Now since $D_6(\varphi)$ is determinantal and of the expected codimension $4$, it is Cohen--Macaulay \cite[Proposition~4.1]{ACGH}.
In other words, it has no embedded component.
It suffices to show that $D_6(\varphi)$ has the same degree as $S^{[2]}$, which follows from a Chern class computation using the Porteous' formula [\loccit, Proposition~4.1]
(the Chern classes of the spinor bundle $\cS_+$ can be found in \cite[Remark~2.9]{ottaviani}).
\end{proof}

Thus, when we consider the restriction of the map $\varphi$ to $S^{[2]}$, we obtain two vector bundles of rank $2$, namely, the kernel bundle $\cE_1\coloneqq\ker(\varphi|_{S^{[2]}})$ and the cokernel bundle $\cE_2\coloneqq\coker(\varphi|_{S^{[2]}})$.

\begin{prop}
\label{prop:normal-bundle}
The normal bundle $\cN_{S^{[2]}/Q}$ admits the following decomposition
\[
\cN_{S^{[2]}/Q}\cong \cE_1^\vee\otimes \cE_2.
\]
\end{prop}
\begin{proof}
We provide a brief explanation for this fact:
the description of $S^{[2]}$ as the degeneracy locus $D_6(\varphi)$ implies that we can factor the inclusion $S^{[2]}\into Q$ as
\[
\begin{tikzcd}
S^{[2]} \ar[r]\ar[dr] & F\coloneqq \Gr_Q(2, \cS_+) \ar[d] \\
                & Q
\end{tikzcd}
\]
where the morphism $S^{[2]}\to F$ sends a point $\xi\in S^{[2]}$ to the pair $(\xi,\ker(\varphi_\xi))\in F$.
The map $\varphi$ (regarded as a global section of $\cS_+^\vee\otimes W$ on $Q$) induces a global section of the vector bundle ${\cU^\vee}\otimes W$ on $F$, where $\cU$ is the rank-$2$ universal subbundle; notably, $S^{[2]}$ is exactly the zero locus of this induced global section.
This gives us a short exact sequence
\begin{equation}
\label{eq:relative-normal}
0\to \cT_{F/Q}|_{S^{[2]}} \to \cN_{S^{[2]}/F} \to \cN_{S^{[2]}/Q} \to 0,
\end{equation}
where $\cN_{S^{[2]}/F}\cong {\cU^\vee}|_{S^{[2]}}\otimes W$.
Note that $\cT_{F/Q}\cong \cU^\vee\otimes \cQ$, where $\cU$ and $\cQ$ are the rank-$2$ universal subbundle and the rank-$6$ universal quotient bundle on $F$, respectively.
By the degeneracy locus description, since $\cU|_{S^{[2]}}$ is the rank-$2$ kernel bundle $\cE_1$ by construction, the quotient $\cQ|_{S^{[2]}}\cong \cS_+/\cU$ can be identified with the image of $\varphi$, thus we may describe the rank-$2$ cokernel bundle $\cE_2$ as $(W\otimes\cO_{S^{[2]}})/\cQ|_X$. Now from the exact sequence \eqref{eq:relative-normal} we read that $\cN_{S^{[2]}/Q}\cong \cE_1^\vee\otimes \cE_2$.
\end{proof}

From the geometric interpretation of the map $\varphi$, the projectivized kernel bundle $\bP (\cE_1)$ is the \emph{universal secant line} to the $K3$ surface $S$.
It follows that $\cE_1$ equals the pullback to $S^{[2]}$ of the universal subbundle on $\Gr(2,H^0(L)^\vee)$ under the natural map $S^{[2]}\into \Gr(2,H^0(L)^\vee)$.
In other words, $\cE_1^\vee$ is the tautological rank-$2$ bundle $L^{[2]}$ induced by $L$ and it has first Chern class $L_2-2\delta$.
Thus the rank-$2$ cokernel bundle $\cE_2$ has first Chern class $3L_2-7\delta$.

\begin{rem}
\label{rem:S2-iso-fm-partner}
The degeneracy locus construction also gives a geometric interpretation of the isomorphism between $S^{[2]}$ and $\cM^{[2]}$: a point $x$ on $S^{[2]}$ is the same as $\bP(\ker\varphi_x)$ which is a secant line to $S$. It naturally has a $2$-dimensional cokernel which is a line in the dual $\bP W^\vee$ where the Fourier--Mukai partner $\cM$ lives.
Thus we obtain a secant line to $\cM$ hence a point on $\cM^{[2]}$.
By symmetry, we see that the two bundles are interchanged on $\cM^{[2]}$, so the class $L_2-2\delta$ is mapped to $\hat L_2-2\hat \delta$.
\end{rem}

\subsection{(Failure of) projective normality}
Since $S^{[2]}$ can be realized as a degeneracy locus, we can study its ideal using the general machinery of determinantal ideals.
Notably, we have the Lascoux resolution which provides a minimal free resolution.
In the particular case of the ideal generated by $(n-1)$-minors of an $n\times n$-matrix, the Lascoux resolution specializes to the \emph{Gulliksen--Negård complex} (\cite{gulliksen-negard}, see also \cite[6.1.8]{weyman}).
\begin{thm}[Gulliksen--Negård complex]
Let $\varphi\colon \cF \to \cG$ be a morphism between two vector bundles of rank $r$ on a variety $X$.
Let $D\coloneqq D_{r-2}(\varphi)$ be the degeneracy locus of rank $r-2$, and assume $D$ has codimension 4 on $X$. Then
there exists the following four-term locally-free resolution for the ideal sheaf of $D$ in $X$
\begin{equation}
\label{eq:GN-complex}
\begin{aligned}
0\to
&(\det\cF)^{\otimes2} \otimes (\det\cG^\vee)^{\otimes2}\to
(\cF\otimes\det\cF)\otimes (\cG^\vee\otimes\det\cG^\vee)\to\\
&(\det\cF\otimes\Lambda \cG^\vee)\oplus (\Lambda\cF\otimes\det \cG^\vee)\to
\bw {r-1}\cF\otimes\bw {r-1} \cG^\vee\to
\cI_{D/X},
\end{aligned}
\end{equation}
where $\Lambda \cF$ and $\Lambda\cG$ are the kernel bundles
\[
\Lambda \cF\coloneqq \ker\left[\cF\otimes\bw{r-1}\cF\to \det\cF\right],\quad
\Lambda \cG\coloneqq \ker\left[\cG\otimes\bw{r-1}\cG\to \det\cG\right].
\]
\end{thm}

We recall the following properties of spinor bundles \cite{ottaviani}.
\begin{prop}
\label{prop:spinor-properties}
Let $\cS$ be a spinor bundle on a smooth quadric $Q$ of dimension $n$.
\begin{enumerate}
\item\label{prop:spinor-properties-1}
 When $n$ is odd or a multiple of $4$, we have $\cS^\vee\cong \cS(1)$;
otherwise, $\cS_+^\vee\cong \cS_-(1)$.
\item For $n\ge 3$, $\det \cS \cong\cO(-2^{[(n-3)/2]})$.
\item For all $d\in\bZ$ and $0 < i < n$, $H^i(Q,\cS(d))=0$.
\item\label{prop:spinor-properties-4} For all $d\le 0$, $H^0(Q,\cS(d))=0$.
\end{enumerate}

\end{prop}
Combining \eqref{prop:spinor-properties-1} and \eqref{prop:spinor-properties-4} and applying Serre duality, we have for all $d\ge 1-n$, $H^n(Q, \cS(d))=0$.

Returning to our case of a smooth quadric $Q$ of dimension $8$.
To simplify notation, we drop the $+$-sign and denote the spinor bundle simply by $\cS$.
We have $\det \cS\cong\cO(-4)$ and $\bw7\cS\cong \cS^\vee(-4)\cong \cS(-3)$.
Following the notation of \cite{mukai-models} (that is, the Bourbaki notation), we denote by $\set{\alpha_i}_{1\le i\le 5}$ the fundamental weights for the group $\SO(10)$, and by $\cE_{\alpha_i}$ the corresponding globally generated homogeneous vector bundle of highest weight $\alpha_i$.
Note that $\cS^\vee$ is precisely $\cE_{\alpha_4}$.
\begin{lem}
\label{lemma:Lambda-S}
We have the following decomposition
\[
0\to \cE_{2\alpha_4}\to \Lambda \cS(5)\to \cE_{\alpha_3}\to 0.
\]
In consequence, for $d\ge 0$ and $i>0$, $H^i(Q, \Lambda \cS(d))$ vanishes except in the following cases
\[
H^4(Q, \Lambda\cS)=H^4(Q,\cE_{2\alpha_4}(-5))=\bC,\quad
H^2(Q, \Lambda\cS(2))=H^2(Q,\cE_{\alpha_3}(-3))=\bC.
\]
\end{lem}
\begin{proof}
By definition $\Lambda \cS$ fits in the following exact sequence
\[
0\to\Lambda \cS\to\cS\otimes \bw7\cS\to \cO(-4)\to 0.
\]
Twisting it with $\cO(5)$, we get
\[
0\to\Lambda \cS(5)\to{\cS^\vee}^{\otimes 2}=\cE_{\alpha_4}^{\otimes 2}\to \cO(1)=\cE_{\alpha_1}\to 0.
\]
We see that this corresponds to the exact sequence of $\SO(10)$-modules
\[
0\to H^0(Q, \Lambda \cS(5))\to V_{\alpha_4}^{\otimes 2}\to V_{\alpha_1}\to 0,
\]
and hence
\[
H^0(Q, \Lambda \cS(5))=V_{2\alpha_4}\oplus V_{\alpha_3}.
\]
In other words, $\Lambda \cS(5)$ is globally generated, and is an extension of the two globally generated homogeneous vector bundles $\cE_{2\alpha_4}$ and $\cE_{\alpha_3}$.
Now we can use the Borel--Weil--Bott theorem to check the claim on the cohomology groups.
\end{proof}

Now we prove that the embedding $S^{[2]}\subset \bP^9$ satisfies normality in degree $d$ for $d\ge3$.
We also show that $Q$ is the unique quadric hypersurface containing $S^{[2]}$.

\begin{prop}
The following hold:
\begin{enumerate}
    \item For every $d\ge 3$ and $i\ge 1$, the higher cohomology group $H^i(Q, \cI_{S^{[2]}/Q}(d))$ vanishes.
In particular, $S^{[2]}$ in $\bP^9$ is $d$-normal for every $d\ge3$.
    \item When $d=2$, we have $H^0(Q, \cI_{S^{[2]}/Q}(2))=0$ and $H^1(Q, \cI_{S^{[2]}/Q}(2))=\bC$.
\end{enumerate}
\end{prop}
\begin{proof}
Consider the Gulliksen--Negård complex \eqref{eq:GN-complex} twisted by $\cO(d)$,
whose corresponding spectral sequence computes the hypercohomology and therefore the cohomology of $\cI(d)\coloneqq\cI_{S^{[2]}/Q}(d)$.
The resolution takes the following form
\begin{equation}
\label{eq:GN-genus-7}
0\to\cO(d-8)\to \cS(d-4)^{\oplus 8} \to \Lambda \cS(d) \oplus \cO(d-4)^{\oplus 63}\to \cS(d-3)^{\oplus 8}\to \cI(d).
\end{equation}
Combining \autoref{prop:spinor-properties} and \autoref{lemma:Lambda-S}, all higher cohomology groups appearing in the resolution \eqref{eq:GN-genus-7} are known:
when $d\ge 0$, the only non-vanishing terms are when $d=0$ where $h^4(Q,\Lambda \cS)=1$, and when $d=2$ where $h^2(Q,\Lambda\cS(2))=1$.
Thus, we deduce from the spectral sequence that $H^i(Q, \cI_{S^{[2]}/Q}(d))$ vanishes for all $d\ge 3$ and $i\ge 1$,
while $h^1(Q, \cI_{S^{[2]}/Q}(2))= 1$.
(Note that the non-vanishing of $H^4(Q, \Lambda \cS)$ is not accidental:
it contributes to $H^3(Q, \cI)\cong H^2(\cO_{S^{[2]}})$,
which is precisely $1$-dimensional since $S^{[2]}$ is symplectic.)
\end{proof}

The Gulliksen--Negård complex also allows us to determine the generators for the homogeneous ideal of $S^{[2]}$ in $\bP^9$.

\begin{prop}
\label{prop:65-quartics}
The ideal of $S^{[2]}$ in $Q$ is generated by $65$ quartics.
\end{prop}
\begin{proof}
It suffices to examine the complex \eqref{eq:GN-genus-7} and its spectral sequence.
Recall that, as soon as $d\ge 3$, all higher cohomology groups appearing in the spectral sequence are zero.
Therefore, we obtain an exact sequence of global sections
\[
\begin{gathered}
0\to H^0(Q, \cO(d-8))\to H^0(Q,\cS(d-4))^{\oplus 8} \to H^0(Q,\Lambda \cS(d)) \oplus H^0(Q,\cO(d-4))^{\oplus 63}\\
\to H^0(Q,\cS(d-3))^{\oplus 8}\to H^0(Q,\cI(d))\to 0.
\end{gathered}
\]
When $d=3$, the first four terms all vanish, hence there are no cubic equations.
When $d=4$, we get
\begin{equation}
\label{eq:65-quartics}
0\to H^0(Q, \cO_Q)^{\oplus 63} \to H^0(Q, \cS(1))^{\oplus 8}=W\otimes  H^0(Q, \cS(1)) \to H^0(Q, \cI(4))\to 0.
\end{equation}
Note that $H^0(Q, \cS(1))=H^0(Q, \cS^\vee)=V_{\alpha_4}$ is the space of spinors $V_+$ which is of dimension~$16$.
This provides us with $65$ quartics as desired.
Finally, for $d\ge 4$, consider the following commutative diagram
\[
\begin{tikzcd}
H^0(Q, \cS(d-3))^{\oplus 8}\otimes H^0(Q, \cO(1))\ar[r]\ar[d,twoheadrightarrow] & H^0(Q, \cS(d-2))^{\oplus8}\ar[d,twoheadrightarrow]\\
H^0(Q, \cI(d))\otimes H^0(Q, \cO(1))\ar[r]                                         & H^0(Q, \cI(d+1)).
\end{tikzcd}
\]
Note that the first row is the natural map
\[
m\colon H^0(Q,\cS(d-3))\otimes H^0(Q,\cO(1))\to H^0(Q,\cS(d-2))
\]
tensored by $W$. In particular, $m$ is also a map of $\SO(10)$-modules, with the group $H^0(Q, \cS(n+1))$ being the irreducible $\SO(10)$-module $V_{n\alpha_1+\alpha_4}$. So $m$ is necessarily surjective.
Hence the first row is surjective, and so is the second row.
\end{proof}

\begin{rem}\label{distquartic-genus7}
If one keeps track of the vector space $W$ appearing in the trivial vector bundle $W\otimes \cO_Q$ in the exact sequence \eqref{eq:65-quartics}, there exists an extension
\[
0\to\bC\to H^0(Q, \cI(4))\to H^0(L)\otimes W \to 0.
\]
The distinguished quartic is given by the first degeneracy locus $D_7(\varphi)$.
The $64$-dimensional cokernel is the Zariski tangent space to $\Gr\left(8, V_+\right)$ at the point $[H^0(L)^\vee]$.
\end{rem}

\subsection{The minimal resolution}
Denote by $(X,H)=(S^{[2]},L_2-2\delta)$ the Hilbert square of a very general $K3$ surface $S$ of genus $7$. 
Our next goal is to deduce a minimal free resolution for $R_X\coloneqq\bigoplus_{d\geq0}H^0(X,H^d)$ in terms of the symmetric algebra $S_X\coloneqq \Sym^\bullet H^0(X, H)$.
Consider the natural map $m\colon S_X\to R_X$ given piece-wise by
\[
m_d\colon S_d=\Sym^d H^0(X, H)\to R_d=H^0(X, H^d).
\]
We have $\ker(m_d)=I_d$, the degree-$d$ piece of the homogeneous ideal of $X$ in $\bP^9$.
From the last section, we know that $m_d$ is surjective for any $d\ne 2$.
Let $q\in S_2$ denote one equation for the unique quadric $Q\coloneqq Z(q)$ containing $X$, then we have $\ker (m_2)=I_2=\bC q$, while $\ker(m_3)=I_3=S_1\cdot q$.

Let us consider a quartic polynomial $f\in S_4$ vanishing along $D$ but not along all of $Q$ (recall that $D=Z(\det(\varphi))\in|\cO_Q(4)|$, hence $D=Q\cap Z(f)$ in $\bP^9$). Note that the hypersurface $F\coloneqq Z(f)\subset\bP^9$ may be chosen to be generically smooth along $X$; indeed, if $F$ is singular along all of $X$, it suffices to take $f+A\cdot q$ for some $A\in S_2\setminus\langle q\rangle$.

Throughout this section, we will assume that $X=S^{[2]}$ equals the singular locus of the first degeneracy locus $D\coloneqq D_7(\varphi)\subset Q$, scheme-theoretically.%
\footnote{The scheme structure of the singular locus of $D$ is given by the $2\times 2$ minors of the Jacobian matrix defined by $q$ and $f$. Equivalently, it is given by the 7-th Fitting ideal of the sheaf $\Omega^1_{D/\mathbb{C}}$, c.f.~\cite[Exercise 21.2.E]{rising-sea}.}
This is true if $(S,L)$ is chosen to be general among Picard rank one $K3$ surfaces of genus $7$, thanks to a Bertini-type argument (see for example \cite[Teorema~2.8]{ottaviani-codim}).
Under this assumption, we are going to prove the following:
\begin{thm}
\label{thm:graded-resolution}
The Betti diagram of the $S_X$-module $R_X$ is given by%
\footnote{The four terms that are underlined contain extra syzygies not expected from computing with the Hilbert function. Indeed, we will see in \eqref{eq:betti-diagram} that these do not appear for a general deformation of $S^{[2]}$.}
\begin{equation}
\label{eq:S2-g7-betti-diagram}
\begin{array}{c|rrrrrrrr}
b_{i,j} & 0             & 1             & 2   & 3   & 4             & 5              & 6  & 7 \\
\hline
0       & 1             & .             & .   & .   & .             & .              & .  & . \\
1       & .             & \underline{1} & .   & .   & .             & .              & .  & . \\
2       & \underline{1} & 10            & .   & .   & .             & .              & .  & . \\
3       & .             & 20            & 126 & 190 & 130           & \underline{46} & 10 & 1 \\
4       & .             & .             & .   & .   & \underline{1} & .              & .  & . \\
5       & .             & .             & .   & .   & .             & 1              & .  & . \\
\end{array}
\end{equation}
In other words, the minimal graded free resolution of $R_X$ is of the form
\begin{equation}
\label{eq:graded-resolution}
...\to S_X(-2)\oplus S_X(-3)^{\oplus 10}\oplus S_X(-4)^{\oplus 20}\to S_X\oplus S_X(-2)\to R_X\to 0.
\end{equation}
\end{thm}

Before giving a proof, we need a more detailed analysis of some hypersurfaces in $\bP^9$. We start by describing the intersection of $X$ with the polar hypersurfaces of $F$:

\begin{lem}\label{polars}
    Let $V_{10}\coloneqq H^0(\cO_X(H))^\vee$. Then for any nonzero $x\in V_{10}$, the following hold:

    \begin{enumerate}
        \item\label{polars-injectivity} The cubic polynomial $\cP_x(f)$ defined by polarity satisfies $\cP_x(f)\notin\ker(m_3)$. In particular, $\cP_x(f)\neq0$ and thus $\setmid{\cP_x(f)}{x\in V_{10}}$ is a $10$-dimensional subspace of $S_3$.
        
        \item\label{polars-sing-F} $\Sing(F)\cap X$ is a divisor of class $2H$ on $X$. If we let $s\in H^0(X,H^2)$ be a section such that $Z(s)=\Sing(F)\cap X$, then $s\notin \im(m_2).$
        
        \item\label{polars-m3} There is an equality
        \[
        m_3(\cP_x(f)) = \cP_x(q)\cdot s\in H^0(X,H^3),
        \]
        where $\cP_x(q)\in H^0(X,H)$ is the polar hyperplane, and $s\in H^0(X,H^2)$ is a suitable multiple of the section defined above (hence independent of $x$).
    \end{enumerate}
\end{lem}
\begin{proof}
We start by noting that 
\[
Z(\cP_x(f))\cap F=\{p\in F: \;\text{$p\in F^{\mathrm{smooth}}$ and }[x]\in T_pF\}\cup \Sing(F),
\]
where $T_pF\subset \bP^9$ denotes the projective tangent hyperplane to $F$ at $p$. Therefore, if $X\subset Z(\cP_x(f))$ we have $[x]\in T_pF=T_pQ$ for every $p\in X\setminus \Sing(F)$, where the last equality follows from the fact that $Q$ and $F$ are tangent at $p$ (as $D=Q\cap F$ is singular along $X$). But this implies $X\subset Z(\cP_x(q))$, which is impossible since $\cP_x(q)$ is a nonzero linear form. In particular, the linear map
\[
V_{10}\to S_3,\quad x\mapsto \cP_x(f)
\]
is injective, which completes the proof of \eqref{polars-injectivity}.

To prove \eqref{polars-sing-F}, recall that our assumption $X=\Sing(D)$ implies that $X$ is the schematic intersection in $\bP^9$ of $Q$, $F$, and the 45 quartics defined as the minors of the Jacobian matrix
\begin{equation}
\label{eq:jacobian-matrix}
\begin{pmatrix}
\frac{\partial f}{\partial x_0}  & ... & \frac{\partial f}{\partial x_9} \\[1ex]
\frac{\partial q}{\partial x_0}  & ... & \frac{\partial q}{\partial x_9} 
\end{pmatrix},
\end{equation}
where $x_0,...,x_9$ denotes any basis of $H^0(X,H)$. Indeed, this matrix dropping rank at a point of $D$ means that $F$ and $Q$ are tangent there.

Thus for a given point $p\in X$ and $i\in\{0,...,9\}$, the vanishing $\frac{\partial q}{\partial x_i}(p)=0$ implies that $\frac{\partial f}{\partial x_i}(p)=0$ (as smoothness of $Q$ yields $\frac{\partial q}{\partial x_j}(p)\neq0$ for some $j$). Conversely, the vanishing $\frac{\partial f}{\partial x_i}(p)=0$ implies that either $\frac{\partial q}{\partial x_i}(p)=0$ or $\frac{\partial f}{\partial x_j}(p)=0$ for all $j$. It follows that there is an equality of schemes
\[
\bigl(Z(\cP_x(q))\cap X\bigr)\;\cup\;\bigl(\Sing(F)\cap X\bigr) \; = \; Z(\cP_x(f))\cap X
\]
for any nonzero $x\in V_{10}$.
In particular, we see that $\Sing(F)\cap X$ is a divisor on $X$ given by $Z(s)$ for some $s\in H^0(X, H^2)$, and the equality in \eqref{polars-m3} holds up to a scalar.

To conclude \eqref{polars-m3}, we need to show that the scalar is independent of $x$.
We fix a linear form, say $x_0$, and up to scaling $s$, we obtain the equality
\[
m_3\left(\tfrac{\partial f}{\partial x_0}\right)=\tfrac{\partial q}{\partial x_0}\cdot s.
\]
Then for any $x_i$, since the minor $\frac{\partial f}{\partial x_i} \cdot \frac{\partial q}{\partial x_0} - \frac{\partial q}{\partial x_i}\cdot \frac{\partial f}{\partial x_0}$ vanishes along $X$, we get
\[
\left(m_3\left(\tfrac{\partial f}{\partial x_i}\right)-\tfrac{\partial q}{\partial x_i}\cdot s\right)\cdot \tfrac{\partial q}{\partial x_0}=0\in H^0(X,H^4).
\]
This allows us to conclude the equality for all $x\in V_{10}$.

Finally, we prove that the section $s$ does not lie in the image of the  map $m_2$. To this end, assume $s=m_2(A)$ for a quadratic polynomial $A\in S_2$. Then, for every $i$, the equality $m_3\left(\frac{\partial f}{\partial x_i}\right)=\frac{\partial q}{\partial x_i}\cdot m_2(A)=m_3\left(A\cdot \frac{\partial q}{\partial x_i}\right)$ implies that
\[
\tfrac{\partial f}{\partial x_i}=A\cdot \tfrac{\partial q}{\partial x_i}+q\cdot H_i
\]
for some linear form $H_i\in S_1$, in virtue of the description that $\ker(m_3)=S_1\cdot q$. In particular, by Euler's identity for homogeneous polynomials, we have
\[
4f=\sum_{i=0}^9\tfrac{\partial f}{\partial x_i}\cdot x_i
=A\cdot 2q + q\cdot \left(\sum_{i=0}^9 H_ix_i\right).
\]
Hence the polynomial $f$ would be divisible by $q$, which is a contradiction.
\end{proof}

In the sequel, we choose homogeneous coordinates $x_0,...,x_9$ in $\bP^9$ so that $q=x_0^2+...+x_9^2$. For every $0\leq i<j\leq 9$, let us consider the quartic polynomial
\[
f_{ij}\coloneqq\tfrac{\partial f}{\partial x_i}\cdot x_j-\tfrac{\partial f}{\partial x_j}\cdot x_i\in I_4=\ker(m_4).
\]
As explained in the proof of \autoref{polars}, up to scalars, these are precisely the $45$ minors of the Jacobian matrix \eqref{eq:jacobian-matrix}, hence $X$ is the intersection in $\bP^9$ of $Q$, $F$, and the hypersurfaces $Z(f_{ij})$. Moreover, we have:

\begin{prop}\label{45-quartics}
    The $45$ (classes of) quartics $f_{ij}$ are linearly independent in $S_4/ (S_2\cdot q)$.
\end{prop}
\begin{proof}
For simplicity, throughout this proof we will denote by $P$ the ambient projective space $\bP^9$. Let $\beta\colon\Omega_{P}|_Q(2)\to \cI_{X/Q}(4)$ be the morphism of $\cO_Q$-modules defined as the composition
\[
\Omega_{P}|_Q(2)\into V_{10}\otimes \cO_Q(1)\into H^0(\cO_Q(3))\otimes \cO_Q(1)\to \cO_Q(4),
\]
where the map $\Omega_{P}|_Q(2)\hookrightarrow V_{10}\otimes \cO_Q(1)$ is induced by the twisted Euler sequence in $\bP^9$ (plus the natural isomorphism $V_{10}^\vee\cong V_{10}$ defined by $q$), and the mid morphism is given by polarity with respect to $f$. It is easy to check (using that $Q$ and $F$ are tangent along $X$) that the image of $\beta$ is indeed contained in $\cI_{X/Q}(4)\subset \cO_Q(4)$. 

On the other hand, there is a natural isomorphism $\cN_{X/Q}\cong \cN_{X/Q}^\vee(4)$ of $\cO_X$-modules.
Indeed, by \autoref{prop:normal-bundle}, we have $\cN_{X/Q}\cong \cE_1^\vee\otimes\cE_2$, where $\cE_1=\ker(\varphi|_X)$ and $\cE_2=\coker(\varphi|_X)$.
Therefore, we get
\[
\cN_{X/Q}\cong \cE_1^\vee\otimes\cE_2\cong \cE_1\otimes \det(\cE_1)^{-1}\otimes \det(\cE_2)\otimes \cE_2^\vee\cong \cN_{X/Q}^\vee(4).
\]
This natural isomorphism can be phrased more geometrically.
Namely, since we have assumed that $\Sing(D)=X$, where $D=D_7(\varphi)$ is defined by the quartic $\det\varphi$ in $Q$, the tangent cone to $D$ along $X$ provides us with a smooth quadric bundle over $X$, which can be equivalently expressed in terms of a pairing
\[
\Sym^2\cN_{X/Q}\to \cO_X(4)
\]
that is nowhere degenerate.

When we combine this description with the natural isomorphism $\cT_Q\cong\Omega_Q(2)$ given by the second fundamental form%
\footnote{This is a pairing $\Sym^2\cT_{Q}\to \cO_Q(2)$, fiberwise given by restriction of $Q$ to its Zariski tangent space.} \cite{griffiths-harris}, an elementary local computation shows that
\begin{equation}\label{fuc-diagram}
    \begin{tikzcd}
\Omega_P|_Q(2) \ar[r,twoheadrightarrow]\ar[d,twoheadrightarrow,"\beta"] & \Omega_Q(2) \ar[r,"\cong"] & \cT_Q\ar[d,twoheadrightarrow] \\
\cI_{X/Q}(4) \ar[r,twoheadrightarrow] & \cN_{X/Q}^\vee(4)\ar[r,"\cong"] &\cN_{X/Q}
\end{tikzcd}
\end{equation}
is a commutative diagram of $\cO_Q$-modules.

Recall that the map $\beta$ factors through the inclusion $\Omega_P(2)|_Q\into V_{10}^\vee\otimes\cO_Q(1)$. Now we consider each $\alpha_{ij}\coloneqq\frac{\partial f}{\partial x_i}\otimes x_j-\frac{\partial f}{\partial x_j}\otimes x_i$ as a global section of $V_{10}^\vee\otimes\cO_Q(1)$ (i.e.~$\alpha_{ij}\in V_{10}^\vee\otimes V_{10}^\vee$). In this way, $\langle \alpha_{ij}\rangle$ is identified with the subspace $\bw 2 V_{10}^\vee\subset V_{10}^\vee\otimes V_{10}^\vee$ of skew-symmetric tensors, so it follows from the short exact sequence
\[
0\to \Omega_P|_Q(2)\to V_{10}^\vee\otimes\cO_Q(1)\to \cO_Q(2) \to 0
\]
that $\langle \alpha_{ij}\rangle\subset H^0(Q,\Omega_P|_Q(2))$.

Note that all we want to prove is that the subspace $\langle\alpha_{ij}\rangle$ does not lose dimension when mapped to $H^0(\cI_{X/Q}(4))$ under $\beta$. According to \eqref{fuc-diagram}, and using $H^0(\cT_X)=0$, this amounts to the injectivity of the map
$\langle\alpha_{ij}\rangle\into H^0(\Omega_P|_Q(2))\to H^0(\cT_Q)$.

In order to check this injectivity, we take the commutative diagram with exact rows and columns
\[
\begin{tikzcd}
 & & 0\ar[d] & 0\ar[d] & \\
 0\ar[r] & \cO_Q\ar[r]\ar[d,"="] & \ker(\gamma)\ar[r]\ar[d] & \cT_Q \ar[r]\ar[d] & 0\\
0\ar[r] & \cO_Q\ar[r] & V_{10}\otimes\cO_Q(1)\ar[r]\ar[d,"\gamma"] & \cT_P|_Q \ar[r]\ar[d] & 0\\
 & & \cO_Q(2)\ar[d]\ar[r,"="] & \cO_Q(2)\cong\cN_{Q/P}\ar[d] & \\
 & & 0 & 0 & 
\end{tikzcd}
\]
where the mid row is the restriction to $Q$ of the (dual) Euler sequence in $\bP^9$. One immediately checks that, upon identifying $V_{10}\cong V_{10}^\vee$ by $q$, we have $\ker(\gamma)\cong\Omega_P|_Q(2)$ and that the map $\ker(\gamma)\to \cT_Q$ gets identified with the first row of \eqref{fuc-diagram}.

Moreover, it is well known that the dual Euler sequence yields
\[
V_{10}\otimes H^0(\cO_Q(1))=\mathrm{End}(V_{10}^\vee)\onto H^0(\cT_P|_Q)=\mathrm{End}(V_{10}^\vee)/\langle \mathrm{Id}\rangle,
\]
whereas the natural inclusion of the 45-dimensional vector space $H^0(\cT_Q)\subset H^0(\cT_P|_Q)$ can be realized by the Lie algebra $\mathfrak{so}(10)$ of skew-symmetric matrices.
From this description, it follows that the subspace $\langle \alpha_{ij}\rangle=\bw 2 V_{10}^\vee$ does not lose dimension when mapped to $H^0(\cT_Q)$, which finishes the proof.
\end{proof}

Now we have all the necessary ingredients to prove \autoref{thm:graded-resolution}:

\begin{proof}[Proof of \autoref{thm:graded-resolution}]
Since $X$ fails quadratic normality, the map $m\colon S_X\to R_X$ is not surjective in degree $2$, so we may take any $s\in R_2\setminus \im(m_2)$ and consider
\[
 \left(m,-s\cdot m(-2)\right)\colon S_X\oplus S_X(-2)\onto R_X,
\]
where $m(-2)$ is the same map $m$ with degree shifted by $-2$.
This gives us the first term of the resolution.

Now, we need to find generators for the kernel $K$ of the previous map. This is the graded $S_X$-submodule of $S_X\oplus S_X(-2)$ with pieces
\[
K_d=\setmid{(G,H)\in S_d\oplus S_{d-2}}{m_d(G)=m_{d-2}(H)\cdot s},
\]
which clearly contains $\ker(m_d)\oplus\set{0}=I_d\oplus\set{0}$.
Since $s\notin\im(m_2)$, we can easily deduce that
\begin{align*}
K_2&=\setmid{(\lambda q,0)}{\lambda\in \bC},\\
K_3&=\setmid{(G,H)\in S_3\oplus S_1}{m_3(G)=H\cdot s}.
\end{align*}
The descriptions of $K_2$ and $K_3$ tell us that to generate $K$ we need at least the map
\begin{equation}\label{partial-generation}
S_X(-2)\oplus S_X(-3)^{\oplus 10}\to S_X\oplus S_X(-2)
\end{equation}
given by the matrix
\[
\begin{pmatrix}
q & f_0 & ... & f_9\\
0 & x_0 & ... & x_9
\end{pmatrix},
\]
where $f_i\in S_3$ is a cubic such that $m_3(f_i)=x_i\cdot s$. As proven in \autoref{polars}, if we take a suitable multiple of $s\in R_2$ defining the divisor $\Sing(F)\cap X$, we can pick $f_i=\frac{\partial f}{\partial x_i}$.
Since two such choices $f_i,f_i'$ differ by a multiple of $q$ (as $\ker(m_3)=S_1\cdot q$), the map \eqref{partial-generation} surjects onto $K$ in degrees $2$ and $3$.

Now we observe that the projection map
\[
\mathrm{pr}_d\colon K_d\into S_d\oplus S_{d-2}\to S_{d-2},
\]
is surjective for all $d\ge 3$, since $K$ has generators of the form $(f_i,x_i)$ in degree $3$. The kernel of $\mathrm{pr}_d$ equals $K_d\cap (S_d\oplus\set{0})=I_d\oplus \set{0}$. On the other hand, we have proven in \autoref{prop:65-quartics} that the homogeneous ideal of $X$ in $\bP^9$ is generated by polynomials of degree $\le 4$. Therefore, as soon as we add to \eqref{partial-generation} all the necessary quartic generators for $I_4$, they will generate $K_d$ in every degree $d\ge 5$.
So it suffices to determine the remaining generators in degree $4$.

With the generators we have already picked, the map \eqref{partial-generation} in degree $4$ looks like
\[
\begin{aligned}
S_2\oplus S_1^{\oplus 10}&\to S_4\oplus S_2\\
\left(A,\left(\textstyle\sum_j\lambda_{ij}x_j\right)_i\right)&\mapsto \left(Aq + \sum_{i,j}\lambda_{ij}f_ix_j,\sum_{i,j}\lambda_{ij}x_ix_j\right).
\end{aligned}
\]
If we impose vanishing of the second component (i.e.~lying in $\ker(\mathrm{pr}_d)$), we get $\lambda_{ij}+\lambda_{ji}=0$ for any $i$ and $j$.
In other words, we obtain elements of the form
\[
\left(Aq + \sum_{i<j}\lambda_{ij}\cdot (f_ix_j-f_jx_i),0\right)\in I_4\oplus \{0\}=\ker(\mathrm{pr}_4),
\]
so it follows from \autoref{45-quartics} that the image of \eqref{partial-generation} intersects $\ker(\mathrm{pr}_4)$ in a 100-dimensional vector space. On the other hand, recall that $\dim I_4=120$ by \eqref{eq:hilbert-function}, as $X$ satisfies quartic normality in $\bP^9$. Therefore, \eqref{partial-generation} can be completed to the second step of the minimal resolution as follows:
\[
S_X(-2)\oplus S_X(-3)^{\oplus 10}\oplus S_X(-4)^{\oplus 20}\to S_X\oplus S_X(-2).
\]

Once we determined the second term in the resolution, we can easily deduce the rest of the Betti diagram.
Notably, by Green's duality \autoref{thm:green-duality}, the last two rows $j\in\set{4,5}$ (as well as any subsequent rows) are determined via the symmetry
\[
\forall i\in\bZ,\ \forall j\ge 4,\quad b_{i,j}=b_{5-i,5-j}.
\]
The remaining terms in the row $j=3$ can be computed one-at-a-time, since the Hilbert function $h^0(\cI_{X/\bP^9}(d))$ is completely determined via the Riemann--Roch polynomial.
\end{proof}

\subsection{Deformation}
The goal of this section is to show that a general deformation of $S^{[2]}$ is contained in no quadric of $\bP^9$; consequently, it is projectively normal and we will deduce its entire Betti diagram, which establishes \autoref{intro-thmdef}.\eqref{intro-thmdef-1}.
For simplicity, we will denote again by $X$ the Hilbert square $S^{[2]}$ and by $P$ the projective space $\bP^9=\bP\left(H^0(L_2-2\delta)^\vee\right)$.

We begin with some cohomological computations of infinitesimal nature.

\begin{prop}
\label{prop:cohomology-dim}
We have
\begin{enumerate}
\item $h^0(X, \cT_P|_X) = 99$, $h^1(X, \cT_P|_X) = 1$;
\item $h^0(X, \cN_{X/P}) = 119$, $h^1(X, \cN_{X/P}) = 0$;
\item $h^0(X, \cT_Q|_X) = 45$, $h^1(X, \cT_Q|_X) = 2$;
\item $h^0(X, \cN_{X/Q}) = 64$, $h^1(X, \cN_{X/Q}) = 0$.
\end{enumerate}
Moreover, for $H^0(X, \cN_{X/Q})$---the space of first-order deformations of $X$ in $Q$---there is a canonical identification
\[
H^0(X, \cN_{X/Q})\cong H^0(L)\otimes W.
\]
\end{prop}
\begin{proof}
We consider the Euler sequence on $P$ restricted to $X$
\[
0\to \cO_X\to\cO_X(1)^{\oplus 10}\to\cT_P|_X\to 0,
\]
from which we easily deduce $h^i(X, \cT_P|_X)$.
Then we consider the short exact sequence
\[
0\to \cT_X\to \cT_P|_X \to \cN_{X/P}\to 0
\]
with the Kodaira--Spencer map
\[
\delta_{X/P}\colon H^0(X, \cN_{X/P}) \to H^1(X, \cT_X).
\]
The local Torelli theorem \cite[Section~8]{beauville} tells us that $\delta_{X/P}$ is of corank $1$. This allows us to determine the cohomology groups of $\cN_{X/P}$.

Next, we compute $h^i(X, \cT_Q|_X)$.
We begin by noticing that on a smooth quadric hypersurface we have an isomorphism $\Omega_Q(2)\cong \cT_Q$.
Indeed, the second fundamental form (see \cite{griffiths-harris})
\[
I\!I\colon\Sym^2\cT_Q\to \cN_{Q/P}=\cO_Q(2)
\]
is nowhere degenerate, so it identifies $\cT_Q$ as $\cT_Q(-2)^\vee\cong\Omega_Q(2)$.
Then we consider the short exact sequence
\[
0\to \cT_Q\to \cT_P|_Q \to \cO_Q(2)\to 0.
\]
Upon dualizing, twisting with $\cO(2)$, and restricting to $X$, we get
\begin{equation}
\label{eq:kernel-bundle-on-X}  
0\to\cO_X\to \Omega_P(2)|_X\to \Omega_Q(2)|_X\cong\cT_Q|_X\to 0.
\end{equation}
Note again that we have the following twisted Euler sequence on $P$ restricted to $X$
\[
0\to\Omega_P(2)|_X\to \cO_X(1)^{\oplus 10} \to \cO_X(2)\to 0.
\]
Since $X$ is contained in a unique quadric and is not quadratically normal,
it is clear that $h^0(X, \Omega_P(2)|_X)=46$, $h^1(X, \Omega_P(2)|_X)=1$, and higher cohomology groups of $\Omega_P(2)|_X$ vanish.
We can then completely determine the cohomology groups of $\cT_Q|_X$ using \eqref{eq:kernel-bundle-on-X}.

Finally, we study the normal bundle $\cN_{X/Q}$.
Recall from \autoref{prop:normal-bundle} that $\cN_{X/Q}\cong \cE_1^\vee\otimes\cE_2$, where $\cE_1=\ker(\varphi|_X)$ and $\cE_2=\coker(\varphi|_X)$.
Moreover, the dual $\cE_1^\vee$ of the kernel bundle is the tautological rank-$2$ bundle $L^{[2]}=\pi_*\psi^* L$ induced by $L$,
where $\pi$ and $\psi$ are the maps given in the diagram \eqref{eq:universallenght2subscheme}.
To compute the global sections of $\cN_{X/Q}$, we have the following
\[
H^0(X, \cN_{X/Q})=H^0(X, \cE_1^\vee\otimes \cE_2)=H^0(X, \pi_*\psi^*L\otimes \cE_2)=H^0(S, L\otimes \psi_*\pi^*\cE_2).
\]
We claim that $\psi_*\pi^*\cE_2$ is the trivial vector bundle $W\otimes \cO_S$, which we shall prove in \autoref{lem:cokernel-pullpush}.
This gives us the desired identification
\[
H^0(X, \cN_{X/Q})\cong H^0(S, L)\otimes W.
\]
To conclude, we consider the short exact sequence
\[
0\to \cT_X\to \cT_Q|_X \to \cN_{X/Q}\to 0
\]
with the Kodaira--Spencer map
\[
\delta_{X/Q}\colon H^0(X, \cN_{X/Q}) \to H^1(X, \cT_X).
\]
Since $h^0(X, \cN_{X/Q}) = 64$, we deduce that $\delta_{X/Q}$ has rank $19$ and $h^1(X, \cN_{X/Q}) = 0$.
\end{proof}

\begin{rem}
As observed in \autoref{distquartic-genus7}, $H^0(L)\otimes W$ is naturally  the Zariski tangent space to $\Gr(8,V_+)$ at the point $[H^0(L)^\vee]$.
Hence deforming $X$ inside $Q$ is the same as varying the chosen $\bP^7$ in the ambient $\bP^{15}$ of the orthogonal Grassmannian; in other words, it is the same as varying the $K3$ surface.
\end{rem}

The following lemma completes the remaining step in the proof of \autoref{prop:cohomology-dim}.

\begin{lem}
\label{lem:cokernel-pullpush}
Let $\cE_2=\coker(\varphi|_X)$ be the rank-$2$ cokernel bundle of $\varphi$ on $X$.
In the notation of diagram \eqref{eq:universallenght2subscheme}, $\psi_*\pi^*\cE_2$ is isomorphic to the trivial vector bundle $W\otimes \cO_S$.
\end{lem}
\begin{proof}
To show that $\psi_*\pi^*\cE_2\cong W\otimes \cO_S$, we restrict to the fiber $\psi^{-1}(p)$ for each $p\in S$, which is a blow up of the $K3$ surface $S$.
By abuse of notation, we denote by $\tS$ also the subvariety of $S^{[2]}$ parametrizing subschemes with support containing $p$.
It suffices to show that the natural map $W\to H^0(\tS,\cE_2|_{\tS})$ (induced by $W\otimes\cO_X\to \cE_2$) is an isomorphism.

Let us first note that the line bundle $L_2-2\delta$ on $S^{[2]}$ restricts to $\cO_\tS(L-2E)$ on $\tS$, where $E$ is the class of the exceptional divisor. In particular, we have a closed immersion
\[
\tS\into \bP(H^0(\tS,L-2E)^\vee)=\bP\left(H^0(S,L\otimes\cI_p^2)^\vee\right)
\]
given by linear projection of $S\subset\bP^7$ from its projective tangent 2-plane at $p$. Precisely, under the isomorphism $E\cong \cN_{S/\bP^7}^\vee(2)$
\cite[Corollary~4.6]{kuznetsov-2} we have $\bP\left(H^0(S,L\otimes\cI_p^2)^\vee\right)=\bP(E^\vee\otimes k(p))$; the latter is a linear subspace of $\bP(H^0(S^{[2]},L_2-2\delta)^\vee)$ as depicted in the diagram \eqref{eq:diagram-genus7}.

Recall from \autoref{rem:S2-iso-fm-partner} that we have an isomorphism between $S^{[2]}$ and $\cM^{[2]}$, under which $\cE_2$ can be identified as the tautological bundle $\hL^{[2]}=\psi'^*\pi'_*\hL$ on $\cM^{[2]}$, where $\pi'$ and $\psi'$ are the projections from the universal length-2 subscheme $B'$ on $\cM$. 
We denote by $\Sigma$ the preimage of $\tS$ under $\pi'$, which is a branched double cover of $\tS$,
and we denote by $\sigma$ the restriction of the map $\psi'$ to $\Sigma$.
It suffices to show that the natural inclusion $W=H^0(\cM,\hL)\hookrightarrow H^0(\Sigma,\sigma^*\hL)$ is an isomorphism.
\[
\begin{tikzcd}
& B'\ar[ddl,"\psi'"']\ar[ddrr,"\pi'"]  &\Sigma\ar[l]\ar[ddll,"\sigma"]\ar[dr]&&& B\ar[ddll,"\pi"']\ar[ddr,"\psi"]\\
&&& \tS\ar[d,hookrightarrow]\\
\cM &&& \cM^{[2]}\cong S^{[2]} &&& S
\end{tikzcd}
\]
Now we observe that the map $\sigma:\Sigma\to\cM$ is finite of degree 4. Indeed, in view of the isomorphism $\cM^{[2]}\cong S^{[2]}$ of \autoref{subsec:FMpartner}, the fiber $\sigma^{-1}(F)$ for a given $F\in\cM$ consists of the four points $q_1,...,q_4\in \tS\subset S^{[2]}$ with the property $h^0(S,F\otimes\cI_{p,q_i})=2$. To make this precise, consider the evaluation map
$H^0(F\otimes\cI_p)\otimes\cO_S\to F$ with kernel $L^{-1}$; dualizing, one obtains a short exact sequence
\begin{equation}\label{eq:ses-sigma}
0\to F^\vee\to H^0(F\otimes\cI_p)^\vee\otimes\cO_S\to L\otimes \cI_{T_F}\to 0
\end{equation}
where $T_F$ is a finite, length-7 subscheme on $S$ containing the fat point $\underline{\mathop{\mathrm{Spec}}}(\cO_S/\cI_p^2)$. Then the fiber $\sigma^{-1}(F)$ is isomorphic to the complement of $\underline{\mathop{\mathrm{Spec}}}(\cO_S/\cI_p^2)$ in $T_F$.

More geometrically, the inclusion  $H^0(S,F\otimes\cI_p)^\vee\subset H^0(S,L\otimes\cI_p^2)$ defined by \eqref{eq:ses-sigma} can be thought of as a line in $\bP\left(H^0(L\otimes\cI_p^2)^\vee\right)=\bP^4$, which is quadrisecant to $\tS$ along the subscheme $\pi'(\sigma^{-1}(F))$. In this way, $\cM$ parametrizes a family of lines in $\bP^4$ which are quadrisecant to $\tS$. We claim that the rank-2 vector bundle $\cF_p^\vee$ on $\cM$ equals the pullback of the universal bundle under the corresponding morphism $\cM\to \Gr\left(2,H^0(L\otimes\cI_p^2)^\vee\right)$.

To prove this, it suffices to observe that the vector bundle map $\Phi(\cI_p)^\vee\hookrightarrow H^0(\cF_p)\otimes\cO_\cM$ (obtained by applying \autoref{rem-FMkernelbundle} on $\cM$) globalizes the inclusion $H^0(S,F\otimes\cI_p)^\vee\subset H^0(S,L\otimes\cI_p^2)$ for every $F\in \cM$. Indeed, by cohomology and base change we have canonical identifications of:
\begin{itemize}
    \item The fiber of $\Phi(\cI_p)^\vee$ at $F\in\cM$ with $H^0(S,F\otimes\cI_p)^\vee$.
    \item $H^0(\cM,\cF_p)$ with the fiber at $p\in S$ of $\hat{\Phi}(\cO_\cM)=E$, hence with $H^0(S,L\otimes \cI_p^2)$.
\end{itemize}

We can thus consider the projective bundle $\bP_\cM(\cF_p^\vee)$, which can be understood as the universal quadrisecant line to $\tS\subset \bP^4$,
while $\Sigma$ is the incidence variety of the four secant points.
We obtain a commutative diagram
\[
\begin{tikzcd}
\Sigma \ar[r,hookrightarrow]\ar[dr,"\sigma"'] &\bP_\cM(\cF_p^\vee) \eqqcolon Z\ar[d]\\
& \cM
\end{tikzcd}
\]
and our goal is to prove that $H^0(\Sigma, \sigma^*\hL)=H^0(\cM,\sigma_*\sigma^*\hL)$ equals $H^0(\cM, \hL)$.
Note that $\Pic(Z)=\bZ\hL\oplus \bZ\zeta$, where by abuse of notation $\hL$ also denotes the pullback of $\hL$ to $Z$, and $\zeta$ is the relative $\cO(1)$ satisfying the relation $\zeta^2-\hL\cdot \zeta + 5\mathrm{pt}_\cM=0$ (recall that $v(\cF_p^\vee)=(2,-\hL,3)$, in particular $c_2(\cF_p^\vee)=5$).
Furthermore, it follows from the relative Euler sequence that $\omega_Z=\cO_Z( \hL-2\zeta)$.

It is easy to deduce from the intersection numbers
\[
\Sigma\cdot \hL^2=4\cdot \deg_{\hL} \cM=48, \quad \Sigma\cdot \zeta^2=2\cdot \deg (\tS/\bP^4)=16
\]
that $\cO_Z(\Sigma)=\cO_Z(-\hL+4\zeta)$.
Hence by considering the short exact sequence
\[
0\to \cO_Z(2\hL-4\zeta)=\cI_{\Sigma/Z}(\hL)\to \cO_Z(\hL) \to \cO_\Sigma(\hL)\to 0,
\]
it suffices to show that $h^0(Z,\cO_Z(2\hL-4\zeta))=h^1(Z,\cO_Z(2\hL-4\zeta))=0$.
Equivalently, by Serre duality we need $h^2(Z,\cO_Z(-\hL+2\zeta))=h^3(Z,\cO_Z(-\hL+2\zeta))=0$.

To this end, we observe that the pushforward of $\cO_Z(-\hL+2\zeta)$ to $\cM$ equals $\Sym^2\cF_p\;(-\hL)$, and there are no higher direct images. Therefore, there are isomorphisms
\[
H^i\left(Z,\cO_Z(-\hL+2\zeta)\right)\cong H^i\left(\cM,\Sym^2\cF_p\;(-\hL)\right)
\]
for any $i\geq0$, and the required vanishing for $i=3$ is trivial. 
For $i=2$, we simply consider the direct sum decomposition
\[
\cF_p\otimes\cF_p(-\hL)\cong \left(\bw2 \cF_p\oplus \Sym^2\cF_p\right)(-\hL),
\]
which becomes
\[
\cF_p\otimes\cF_p^\vee\cong \cO_\cM\oplus \Sym^2\cF_p\;(-\hL)
\]
in virtue of the equality $\det(F_p)=\hL$. But then $h^2(\cO_\cM)=1$ and $h^2(\cF_p\otimes\cF_p^\vee)=\hom(\cF_p,\cF_p)=1$, which implies the required vanishing $h^2(\Sym^2\cF_p\;(-\hL))=0$ and completes the proof.
\end{proof}

Now we can derive the desired results for a general deformation of $(S^{[2]},L_2-2\delta)$.

\begin{prop}
\label{prop:genus-7-deformation}
A general deformation of $S^{[2]}$, that is, a general member $(X', H')\in\cM_4^{(1)}$, is not contained in any quadric hypersurface of $\bP^9$.
\end{prop}
\begin{proof}
Let $\cH_X$ and $\cH_Q$ denote the irreducible components of the Hilbert schemes containing $[S^{[2]}]$ and quadric hypersurfaces, respectively.
By \autoref{prop:cohomology-dim}, $\cH_X$ is generically smooth of dimension $119$, while $\cH_Q\cong |\cO_P(2)|$ is smooth of dimension $54$.
It is clear that a general deformation of $S^{[2]}$ is contained in at most one quadric; assuming that such quadric exists, we may consider the open subset $U\subset \cH_X$
\[
U\coloneqq\setmid{[X']\in \cH_X}{\text{$X'$ is smooth and contained in a unique quadric}},
\]
and the incidence variety
\[
I\coloneqq\setmid{([X'], [Q'])\in \cH_X\times \cH_Q}{[X']\in U,\; X'\subset Q'}.
\]
The projection $\pr_1\colon I\to \cH_X$ is dominant and injective by definition hence birational.

Now we consider our distinguished point $x\coloneqq ([X], [Q])$ on $I$.
The Zariski tangent space of $I$ at $x$ has dimension $\ge 119$, thus the subspace $\ker \pr_{2*}$ has dimension $\ge 65$.
On the other hand, the space of first-order deformations of $X$ in $Q$ is $H^0(X, \cN_{X/Q})$, which has dimension $64$ by \autoref{prop:cohomology-dim}, a contradiction.
\end{proof}

\begin{cor}
\label{cor:square-4-projective-normal}
A general member $(X', H')\in\cM_4^{(1)}$ is projectively normal. Its homogeneous ideal is generated by $10$ cubics and $20$ quartics, and there are no linear syzygies among the $10$ cubics. In other words, the natural map
\begin{equation}
\label{eq:ideal-multiplication}
m\colon H^0(\bP^9,\cI_{X'/\bP^9}(3))\otimes H^0(\bP^9,\cO(1))\to H^0(\bP^9,\cI_{X'/\bP^9}(4))
\end{equation}
is injective.
Moreover, we have the following Betti diagram of $X'$.
\begin{equation}
\label{eq:betti-diagram}
\begin{array}{c|rrrrrrrr}
b_{i,j} & 0 & 1  & 2   & 3   & 4   & 5  & 6  & 7 \\
\hline
0       & 1 & .  & .   & .   & .   & .  & .  & . \\
1       & . & .  & .   & .   & .   & .  & .  & . \\
2       & . & 10 & .   & .   & .   & .  & .  & . \\
3       & . & 20 & 126 & 190 & 130 & 45 & 10 & 1 \\
4       & . & .  & .   & .   & .   & .  & .  & . \\
5       & . & .  & .   & .   & .   & 1  & .  & . \\
\end{array}
\end{equation}
\end{cor}
\begin{proof}
Using the Riemann--Roch polynomial, we get an equality
\[
h^0(X', H'^2)=55=h^0(\bP^9,\cO(2)).
\]
Since a general member $(X', H')\in\cM_4^{(1)}$ is not contained in any quadric, we get quadratic normality and therefore full projective normality.

Now we study the generators of the ideal of $X'$ in $\bP^9$. On the one hand, $X'$ being projectively normal, the Hilbert function $d\mapsto h^0(\bP^9,\cI_{X'/\bP^9}(d))$ is completely determined: we have $h^0(\bP^9, \cI_{X'/\bP^9}(3))=10$ and $h^0(\bP^9, \cI_{X'/\bP^9}(4))=120$. Therefore, the image of the map $m$ in \eqref{eq:ideal-multiplication} is of codimension at least $20$, so we need at least these number of equations. On the other hand, Betti numbers are upper semicontinuous for a family of polarized hyperkähler manifolds;%
\footnote{This follows, for instance, from the sheaf cohomology interpretation of $\cK_{i,j}(X',H')$ for every $j\geq1$ (see \cite[Proposition 3.2]{ein-lazarsfeld}).}
hence the graded resolution in \eqref{eq:graded-resolution} shows that these numbers are also an upper bound.

Similarly, for the rest of the Betti diagram, we have both a lower bound from the Hilbert function and an upper bound from the semicontinuity of the Betti numbers and the Betti diagram of $S^{[2]}$ in \eqref{eq:S2-g7-betti-diagram}.
This, combined with the symmetry provided by Green's duality theorem, allows us to conclude the full Betti diagram of $X'$. 
\end{proof}

\begin{rem}\leavevmode
\begin{enumerate}
\item In the case of $S^{[2]}$, we see that the ten cubics degenerate into the ten multiples of the distinguished quadric $Q$, and they trivially admit $45$ linear syzygies, explaining the extra $45$ quartics we needed to generate the ideal.
\item For a projective hyperkähler manifold $X$ of dimension $\ge 4$, since $h^2(X,\cO_X)=1$, $X$ can never be embedded in a projective space as an arithmetically Cohen--Macaulay variety. Thus the graded Auslander--Buchsbaum formula predicts that the graded minimal resolution of the coordinate ring of $X$ must have length greater than its codimension.
Indeed, here we see that the graded minimal resolution has length equal to $7$, while the codimension of $X$ in $\bP^9$ is $5$.
\item We prove in \cite{coble-type-hypersurfaces} that a general deformation $X'$ of $S^{[2]}$ is the singular locus of a unique quartic hypersurface in $\bP^9$. In other words, the ten cubics are exactly the polars of this quartic, and $X'$ will be the scheme-theoretical intersection of them. The extra quartics are needed to generate the saturated homogeneous ideal of $X'$ in $\bP^9$.
\end{enumerate}
\end{rem}

\section{Genus 8}
\label{sec:genus-8}

\subsection{Quadrics containing the Mukai model}
Let $(S,L)$ be a polarized $K3$ surface of genus 8, with $\Pic(S)=\bZ\cdot L$. The embedding $S\into \bP(H^0(L)^\vee)=\bP^8$ was described in \cite{mukai-models} as follows. Let $E$ be the unique stable vector bundle with $v(E)=(2,L,4)$. It is globally generated with $h^0(E)=6$, and it defines an immersion
\begin{equation}\label{eq:mukai-model8}
    S\into \Gr(2,H^0(E)^\vee)\into \bP(\bw2 H^0(E)^\vee)
\end{equation}
such that $S$ is the intersection of $\Gr(2,H^0(E)^\vee)$ with the subspace $\bP(H^0(L)^\vee)\subset \bP(\bw2 H^0(E)^\vee)$ (this inclusion being defined by the determinant map $\bw2 H^0(E)\onto H^0(L)$).

The vector space of quadrics $I_S(2)$ equals the vector space of quadrics in $\bP(\bw2 H^0(E)^\vee)$ containing $\Gr(2,H^0(E)^\vee)$; this is canonically identified with $\bw2 H^0(E)^\vee$ (any element in $\bw2 H^0(E)^\vee$ defines by wedge a symmetric bilinear form on $\bw2 H^0(E)^\vee$). We obtain
\[
I_S(2)\cong \bw2 H^0(E)^\vee,
\]
so that $\Gr(4,H^0(E))=\Gr(2,H^0(E)^\vee)\subset \bP(\bw2 H^0(E)^\vee)$ is naturally identified with 
\[
\psi_2:\Gr(4,H^0(E))\overset{\cong}{\to}Y_2 \into |I_S(2)|.
\]
(recall \eqref{eq:def-relGrass}). In this way, \eqref{eq:mukai-model8} can be regarded as an immersion of $S$ into the space of quadrics containing it, admitting the following geometric interpretation.

\begin{prop}\label{tangentquadric8}
The map $S\into \bP(\bw2 H^0(E)^\vee)=|I_S(2)|$ sends a point $p\in S$ to the unique quadric in $|I_S(2)|$ which is singular along the projective tangent $2$-plane $T_pS$.
\end{prop}
\begin{proof}
The quadric is indeed singular along $T_pS=\bP\left(H^0(L\otimes\cI_p^2)^\perp\right)\subset \bP(H^0(L)^\vee)$, since for $V=H^0(E\otimes\cI_p)$ the natural map
\[
\lambda_{E,V}:\bw2 V\into \bw2 H^0(E)\to H^0(L)
\]
satisfies $\im(\lambda_{E,V})\subset H^0(L\otimes\cI_p^2)$. Hence we only need to check that $\psi_2\left(E,H^0(E\otimes\cI_p)\right)$ is the unique quadric in $|I_S(2)|$ which is singular along $T_pS$.

Assume there is another quadric $Q$ with this property. In particular $Q$ is of rank $\leq 6$, so by \autoref{prop:sigma-closure} we have $Q=\psi_i(F,V')$ for some $i\in\{0,1,2\}$ and $F\in\cM(2,L,i+2)$, $V'\in\Gr(4,H^0(F))$ such that $\im(\lambda_{F,V'})\subset H^0(L\otimes \cI_p^2)$. Moreover, replacing $F$ by $F^{**}$ we can assume  that $F$ is locally free. Then one can check that:
\begin{enumerate}
    \item\label{genus8-situation1} $\im(\lambda_{F,V'})\subset H^0(L\otimes \cI_p)$ implies that $V'\cap H^0(F\otimes\cI_p)$ is of dimension $\geq3$.

    \item\label{genus8-situation2} If $V'\cap H^0(F\otimes\cI_p)$ is of dimension $3$, then $\im(\lambda_{F,V'})\subset H^0(L\otimes \cI_p^2)$ implies that every section in $V'\cap H^0(F\otimes\cI_p)$ has zero locus of length $\geq2$ at the point $p$.
\end{enumerate}

If $i=2$, then $F=E$ and $V'\cap H^0(F\otimes\cI_p)$ must be $3$-dimensional by the assumption $V'\neq H^0(E\otimes\cI_p)$. But inside $\bP(H^0(E\otimes\cI_p))=\bP^3$, the locus of sections vanishing with length $\geq 2$ at $p$ forms a non-degenerate hypersurface, which contradicts \eqref{genus8-situation2}.

If $i=1$, we first observe that $h^1(F\otimes\cI_p)=0$; otherwise, we would have a stable extension
\[
0\to F^\vee \to G \to \cI_p\to 0
\]
with $v(G)^2=-4$, which is impossible. This implies $h^0(F\otimes\cI_p)=3$ and $h^0(F)=5$. Note that on the one hand, the vector bundle $F$ is globally generated (this can be checked using the stability of the kernel of the evaluation map $H^0(F)\otimes\cO_S\to F$). On the other hand, by \eqref{genus8-situation2} and \cite[Lemma~2.4]{moretti-rojas} the cokernel of the restricted evaluation map 
\[
H^0(F\otimes\cI_p)\otimes \cO_S \to F
\]
has length $\geq 4$ at $p$, and hence $F$ cannot be globally generated at $p$, which is a contradiction.

Finally, if $i=0$ then $h^1(F)=0$ and therefore $V'=H^0(F)$; it follows from \eqref{genus8-situation1} that $F$ is not globally generated at $p$. Let $U_F$ denote the kernel of the natural map $H^0(F)\otimes\cO_S\to F$. Then $U_F$ is slope stable and, as explained in \cite[Proof of Theorem~3.21]{markman} and \cite{ogrady}, we have $Q=\psi_0(U,H^0(F)^\vee)$ for an appropriate non-locally free sheaf $U\in\cM(2,L,2)$ with $U\subset U_F^*$; replacing $U$ by $U_F^*$ we can again reduce to the case $i=1$ or $i=2$, which finishes the proof. 
\end{proof}

As we proved in \autoref{sec:proof-thm-A} (see \autoref{sepeven}), the complete linear system $|L_2-2\delta|$ on $S^{[2]}$ defines a closed immersion
\[
S^{[2]}\into |I_S(2)|^\vee=\bP(\bw2 H^0(E))
\]
which factors through $\Gr(2,H^0(E))$, via the natural map $\xi\mapsto H^0(E\otimes\cI_\xi)$. Thus we can consider the stratification
\[
S^{[2]}\into \Gamma \into\Gr(2,H^0(E)),
\]
where $\Gamma$ is the locus of $V\in\Gr(2,H^0(E))$ such that $V\subset H^0(E\otimes\cI_p)$ for some $p\in S$. 
Similarly to the case of genus $7$ (see \eqref{eq:diagram-genus7}), there is a commutative diagram
\begin{equation}
\begin{tikzcd}
B \ar[d,"\pi"]\ar[r,hookrightarrow] & \Gr_S(2,K) \ar[d]\ar[dr]   \\
S^{[2]}\ar[r,hookrightarrow]  & \Gamma\ar[r,hookrightarrow]  & \Gr(2,H^0(E))
\end{tikzcd}
\end{equation}
where:
\begin{itemize}
    \item $K$ is the kernel bundle on $S$ sitting in the short exact sequence
    \[
    0\to K\to H^0(E)\otimes\cO_S\to E\to 0.
    \]
    In other words, $K^\vee$ is the restriction to $S$ of the rank-4 tautological quotient bundle on $\Gr(2,H^0(E)^\vee)$.
    
    \item $\Gr_S(2,K)$ is the Grassmannian bundle defined by $K$, and $\Gr_S(2,K)\to\Gr(2,H^0(E))$ is the natural morphism. Its image equals $\Gamma$.

    \item The left square is cartesian, namely $B$ is the universal length-$2$ subscheme of \eqref{eq:universallenght2subscheme}.

    \item The morphism $\Gr_S(2,K)\to\Gamma$ is birational (defines an isomorphism outside $B$) and $\Gamma$ is non-normal, singular along $S^{[2]}$.
\end{itemize}

\subsection{Degeneracy locus}
The Hilbert square $S^{[2]}$ can again be realized as a degeneracy locus, in a strikingly similar fashion to the case of genus $7$. In the sequel, we will denote by $G$ the Grassmannian $G\coloneqq\Gr(2,H^0(E))$.
Consider the incidence variety
\[
\begin{tikzcd}
 &I\ar[ld, "p"']\ar[rd, "q"]\coloneqq\setmid{([U],[V])}{(U,V)=0}\\
G\coloneqq\Gr(2,H^0(E))&&\Gr(2,H^0(E)^\vee)
\end{tikzcd}
\]
Note that $I$ is nothing but the flag variety $\Fl(2,4,H^0(E))$.
Each fiber of $p$ is a 4-dimensional quadric in $\Gr(2,H^0(E)^\vee)$, spanning a $\bP^5$ in $\bP(\bw2 H^0(E)^\vee)$. This construction globalizes to the rank-6 vector bundle $\bw2\cQ^\vee$ (where $\cQ$ is the rank-4 tautological quotient on $G$), which is naturally a subbundle of the trivial vector bundle $\bw2 H^0(E)^\vee\otimes\cO_G$.
We denote it by $M\coloneqq \bw2\cQ^\vee$ from now on.
Note that there exists a natural pairing $\bw2\cQ\otimes\bw2\cQ\to \bw4\cQ\cong \cO_G(1)$, and so $M^\vee\cong M(1)$.

On the other hand, the $K3$ surface $S$ is determined by the choice of the $9$-dimensional linear subspace
\[
H^0(L)^\vee\subset \bw2 H^0(E)^\vee,
\]
so we get a quotient map $\bw2 H^0(E)^\vee\onto W$ where $W$ is of dimension $6$.
Via composition, we obtain a morphism between two vector bundles of rank $6$
\[
\varphi\colon M \into\bw2 H^0(E)^\vee\otimes \cO_G\onto W\otimes\cO_G.
\]
The same argument as in \autoref{prop:S2-is-D6} applies, \emph{mutatis mutandis}, to the following result:
\begin{prop}
\label{prop:S2-is-D4}
Scheme-theoretically, $S^{[2]}$ coincides with $D_4(\varphi)$.
\qed
\end{prop}

\begin{rem}
Let us explain how the embedding of $S^{[2]}$ via $|L_2-2\delta|$ studied here is related to the realization of $S^{[2]}$ as the variety of lines of a Pfaffian cubic fourfold studied by Beauville--Donagi in \cite{beauville-donagi}. Recall that the latter is given by $|2L_2-5\delta|$, which has the same square $6$ but divisibility $2$.

The first degeneracy locus $D_5(\varphi)$ admits a rational map $f$ to $\bP W^\vee$ via $x \mapsto \im \varphi_x$.
We can view $\bP W^\vee$ as a subspace of $\bP(\bw2 H^0(E))$.
If we write $x=[V_2]$, then by definition $f(x)$ is orthogonal to $\bw2\cQ^\vee_x=\bw2 (H^0(E)/V_2)^\vee$.
In other words, it is a one-dimensional subspace of $V_2 \wedge H^0(E)$ and therefore cannot be a bivector of maximal rank.
So the image of the rational map $f$ is contained in the intersection of the $5$-dimensional $\bP W^\vee$ with the Pfaffian hypersurface of $\bP(\bw2 H^0(E))$.

On the other hand, the indeterminacy locus of $f$ is precisely the next degeneracy locus $D_4(\varphi)$ that is $S^{[2]}$.
And once the indeterminacy is resolved, each point provides a line in the Pfaffian cubic fourfold, as desired.

Note that the rank-$2$ kernel bundle is the universal secant line to $S$, so it has first Chern class $-(L_2-\delta)$. Thus the rank-$2$ cokernel bundle has first Chern class $2L_2-5\delta$, recovering the computation of Beauville--Donagi.
\end{rem}

\subsection{Projective normality}
This section is concerned with the equations of $S^{[2]}$ in $\bP^{14}=\bP\left(\bw2 H^0(E)\right)$. Keeping in mind that $G=\Gr(2, H^0(E))$ is projectively normal, to establish the projective normality of $S^{[2]}\subset\bP^{14}$ it is sufficient to show the vanishing of $H^1(G, \cI_{S^{[2]}/G}(d))$ for all $d\ge2$.
\begin{prop}
For every $d\ge 2$ and $i\ge 1$, the cohomology group $H^i(G, \cI_{S^{[2]}/G}(d))$ vanishes.
\end{prop}
\begin{proof}
Consider the Gulliksen--Negård complex \eqref{eq:GN-complex} twisted by $\cO(d)$, whose corresponding spectral sequence computes the hypercohomology and therefore the cohomology of $\cI_{S^{[2]}/G}(d)$.
Using the fact that $\det M\cong\cO(-3)$ and $\bw5M\cong M^\vee(-3)\cong M(-2)$, the resolution takes the following form
\[
0\to\cO(d-6)\to M(d-3)^{\oplus 6} \to \Lambda M(d) \oplus \cO(d-3)^{\oplus 35}\to M(d-2)^{\oplus 6}\to \cI(d).
\]
Denoting the four terms by $F_3,\dots,F_0$ respectively, we have a spectral sequence
\[
E_1^{-k, i+k}\coloneqq H^{i+k}(G, F_k)\to H^i(G, \cI(d)).
\]
We claim that as soon as $d\ge 2$, all higher cohomology groups in the spectral sequence vanish:
\begin{equation}
\label{eq:spectral-sequence-vanishing}
\forall d\ge 2, \forall i \ge 1,\quad H^i(G, F_k)=0.
\end{equation}
This immediately implies the required vanishing $H^i(G, \cI(d))=0$. In order to prove \eqref{eq:spectral-sequence-vanishing},
we study the individual terms appearing in the resolution.

\begin{itemize}
\item Using Bott vanishing, we have $H^i(G, \cO(d))=0$ for all $d\ge -5$ and $i\ge 1$.
\item Using the Borel--Weil--Bott theorem, one checks that $H^i(G, M(d))$ vanishes for every $d\ge-5$ and $i\ge 1$.
\item Lastly we study the kernel bundle $\Lambda M$ which fits in the exact sequence
\[
0 \to \Lambda M \to M^{\otimes2}(-2) \to \cO(-3)\to 0.
\]
The tensor square $M^{\otimes2}$ in turn fits in an extension
\[
0 \to \Sym^2M\to M^{\otimes2} \to \bw2 M\to 0.
\]
One verifies that $\bw2 M=\bw2\bw2\cQ^\vee$ can be identified as the irreducible homogeneous vector bundle $\Sigma^{2,1,1}\cQ^\vee$, where $\Sigma^\lambda$ is the Schur functor.
This enables us to use the Borel--Weil--Bott theorem to check that $H^i(\bw2 M(1))=0$ for any $i$,
so we get a surjective map $\Lambda M \onto \bw2 M(-2)$ whose kernel is isomorphic to the kernel of $\Sym^2M(-2)\onto \cO(-3)$.
One also checks that the latter is the irreducible homogeneous vector bundle $\Sigma^{2,2}\cQ^\vee(-2)$;
in other words, we have an exact sequence
\[
0\to \Sigma^{2,2}\cQ^\vee(-2)\to \Lambda M \to\Sigma^{2,1,1}\cQ^\vee(-2) \to 0.
\]
Then using Borel--Weil--Bott again, we check that $H^i(G, \Lambda M(d))=0$ when $d\ge -3$ and $i\ge 1$,
with the exception of
$H^4(G, \Lambda M)=H^4(G, \Sigma^{2,2}\cQ^\vee(-2))=\bC$.
(Again, similar to the genus $7$ case, $h^4(G, \Lambda M)=1$ amounts to $h^2(\cO_{S^{[2]}})=1$.)
\end{itemize}
Combining all these vanishing results, we check that the claim \eqref{eq:spectral-sequence-vanishing} is indeed verified, which concludes the proof.
\end{proof}

\begin{prop}
The ideal of $S^{[2]}$ in $G$ is generated by $55$ cubics.
\end{prop}
\begin{proof}
We have seen that, as soon as $d\ge 2$, all higher cohomology groups appearing in the spectral sequence are zero.
Therefore, we obtain a long exact sequence of global sections
\[
\begin{gathered}
0\to H^0(G, \cO(d-6))\to H^0(G,M(d-3))^{\oplus 6} \to H^0(G,\Lambda M(d)) \oplus H^0(G,\cO(d-3))^{\oplus 35}\\
\to H^0(G,M(d-2))^{\oplus 6}\to H^0(G,\cI(d))\to 0.
\end{gathered}
\]
When $d=2$, the first four terms all vanish, hence there are no quadric equations.
When $d=3$, $H^0(G, M(1))=H^0(G, \bw2\cQ)=\bw2H^0(E)$ is $15$-dimensional, so we get $15\times 6 - 35=55$ cubics as desired.
Finally, for $d\ge 3$, we have the following commutative diagram
\[
\begin{tikzcd}
H^0(G, M(d-2))^{\oplus 6}\otimes H^0(G, \cO(1))\ar[r]\ar[d,twoheadrightarrow] & H^0(G, M(d-1))^{\oplus6}\ar[d,twoheadrightarrow]\\
H^0(G, \cI(d))\otimes H^0(G, \cO(1))\ar[r]                                    & H^0(G, \cI(d+1)).
\end{tikzcd}
\]
Again, similar to the case of genus $7$, the first row is the natural map
\[
m\colon H^0(G,M(d-2))\otimes H^0(G,\cO(1))\to H^0(G,M(d-1))
\]
tensored by $W$, which is also a map of $\GL(2,6)$-modules.
Since $M$ is an irreducible homogeneous vector bundle, the groups $H^0(G, M(n))$ are irreducible as $\GL(2,6)$-module, so $m$ is necessarily surjective.
We may thus deduce the surjectivity of the second row.
\end{proof}

\begin{rem}
Again, by keeping track of the vector space $W$, we have an extension
\[
0\to\bC\to H^0(G, \cI(3))\to H^0(L)\otimes W \to 0,
\]
where the distinguished cubic is given by the first degeneracy locus $D_5(\varphi)$, and the cokernel is the tangent space to $\Gr(9, \bw2H^0(E)^\vee)$ at the point $[H^0(L)^\vee]$.
\end{rem}

\begin{cor}
\label{cor:S2-g8-graded-resolution}
The Betti diagram of the coordinate ring of $S^{[2]}$ in $\bP^{14}$ has the following form
\begin{equation}
\label{eq:S2-g8-betti-diagram}
\begin{array}{c|rrrrrrrrrrrrrrr}
b_{i,j} & 0 & 1  & 2   & 3  & 4 & 5 & 6 & 7  & 8  & 9  & 10 & 11 & 12 & 13 \\
\hline
0       & 1 & .  & .   & .  & . & . & . & .  & .  & .  & .  & .  & .  & .  \\
1       & . & 15 & 35  & 21 & . & . & . & .  & .  & .  & .  & .  & .  & .  \\
2       & . & 55 & 336 & *  & * & * & ? & ?  & ?  & ?  & ?  & ?  & ?  & ?   & \dots \\
3       & . & .  & ?   & ?  & * & * & * & *  & *  & *  & *  & *  & *  & ?   & \dots \\
4       & . & .  & .   & .  & . & . & . & 21 & 35 & 15 & .  & .  & .  & .  \\
5       & . & .  & .   & .  & . & . & . & .  & .  & .  & 1  & .  & .  & .  \\
\end{array}
\end{equation}
where entries marked with an $*$ are known to be non-zero, while those marked with a $?$ could take any value.
\end{cor}
\begin{proof}
Since we have determined the generators, this gives the first two columns $i\in\set{0,1}$.
The $15$ quadrics are the Plücker quadrics defining the Grassmannian $\Gr(2,6)$, and their syzygies are well-known: this gives the first two rows $j\in\set{0,1}$ as well as the non-vanishing of $b_{i,2}$ for $i\in\set{3,4,5}$.
For the rest of the diagram, we use the Hilbert function and the Green's duality theorem as usual.
\end{proof}

Let us finish this section with some comments on \autoref{intro-thmdef}.\eqref{intro-thmdef-2}, concerning a general deformation $(X,H)\in \cM_6^{(1)}$ of $S^{[2]}$ in $\bP^{14}$. First note that $X$ is indeed projectively normal, as $S^{[2]}$ itself is. Using the Riemann--Roch polynomial, we have $h^0(\bP^{14}, \cI_{X/\bP^{14}}(2))=15$ and $h^0(\bP^{14}, \cI_{X/\bP^{14}}(3))=245$, and so the image of the map
\[
H^0(\bP^{14}, \cI_{X/\bP^{14}}(2))\otimes H^0(\bP^{14}, \cO(1))\to H^0(\bP^{14}, \cI_{X/\bP^{14}}(3))
\]
has codimension at least $20$. Therefore, to generate the homogeneous ideal of $X$ we need $15$ quadrics and at least $20$ cubics---note that in the case of $S^{[2]}$, we needed $55$ cubics due to the $35$ linear syzygies among the Plücker quadrics.

We prove in the forthcoming work \cite{coble-type-hypersurfaces} that a general $(X,H)\in \cM_6^{(1)}$ is the singular locus of a unique cubic hypersurface in $\bP^{14}$. In other words, similarly to the case of square $4$, the $15$ quadrics are exactly the polars of this cubic, and $X$ will be the scheme-theoretical intersection of them. The extra cubics are needed to generate the saturated homogeneous ideal.

This would be quite different from the varieties of lines for cubic fourfolds, which have the same square $6$ but divisibility $2$: for these, the $15$ quadrics are always the Plücker quadrics, so they are not sufficient at all to recover the variety.

\section{Final questions and comments}\label{sec:questions}

Apart from the aforementioned work \cite{coble-type-hypersurfaces} on Coble type hypersurfaces, our study naturally raises several questions pointing towards a more systematic understanding of linear systems on hyperkähler fourfolds of $K3^{[2]}$-type. 

In the case of polarized Hilbert squares $(S^{[2]},L_2-2\delta)$, our good understanding of the projective geometry and equations in genus 7 and 8 strongly relies on the existence of a Mukai model for $S$. In particular, we do not have a very satisfactory geometric realization of the multiplication maps
\[
\Sym^d H^0(S^{[2]},L_2-2\delta)\to H^0\left(S^{[2]},(L_2-2\delta)^{\otimes d}\right)
\]
in arbitrary genus. Nevertheless, in view of our results (as well as some supporting computer experiments), it is reasonable to expect:

\begin{conj}\label{conj-hilbertsquare}
    Let $(S,L)$ be a polarized $K3$ surface of genus $g$ with $\Pic(S)=\bZ\cdot L$. Then:
\begin{enumerate}
    \item If $g\geq8$ (that is, $q(L_2-2\delta)\ge6$), then the embedding $S^{[2]}\hookrightarrow \bP\left(H^0(S^{[2]},L_2-2\delta)^\vee\right)$ is projectively normal.

    \item\label{conj-hilbertsquare-2} If $g\geq10$ (that is, $q(L_2-2\delta)\ge10$), then the homogeneous ideal of $S^{[2]}$ is generated by quadrics.
\end{enumerate}
\end{conj}

For general polarized hyperkähler fourfolds of $K3^{[2]}$-type that are not actual Hilbert squares,
the bounds in \autoref{conj-hilbertsquare} are no longer optimal.
Namely, we have showed that projective normality holds also in the case of square $4$, which is indeed predicted by the Hilbert function;
similarly, an elementary computation shows that generation of the homogeneous ideal by quadrics may also be satisfied by hyperkähler fourfolds of $K3^{[2]}$-type equipped with a polarization of square $8$. 
In much greater generality, it is possible to compute a \emph{theoretical Betti diagram} by imposing that the syzygies of a hyperkähler fourfold $(X,H)\in\cM_{2d}^{(\gamma)}$ are as simple as possible (subject to some relations such as Green's duality \autoref{thm:green-duality}). The reader may consult \cite{nazgul} for further details, including an \href{https://sagecell.sagemath.org/?z=eJx1VE1vnDAQvUfKfxjRQ-0F0jXtKSpVxaWVckxuiKwMnt1YxYaAaZr--o5hydINa2kX8MybjzfPVriHEp3TOyv1b2Qmgiddl_z2-gpofYCqMe3gENwTwugCmXcHpeWhkwaGXtvDaP1JOOwc7AdbOd3YKYKFFLR1zEdlgnOIQUyWe7LU0pRKgrqFUtvGaFkzCyGoCBSfvLIIuuYlgoq881wUYb4tNszGCS8ioHf6KybPF0qBUKNlFYevKWVuuvEzo09INiZMjl35pXz2yTeE5LRfRtCTaRvNZfq1p0ia2oBO2gOyCccX0UYkhClhN1DluqDHPdOh4P_7kDmFeBG5hHgCsXumPDkjT4ovcNWNtj0Ry6imki-hFIvFgj9OTZ4smmYK32B7ViDxeCPbFq1iyzhY93jmmc1-BCF6jpSL2CeiLc75u8h-PN5tdqG6PDIuj9Pp0A2dhez6Ss2Sey82J38h_OgQ7cce1CBr7V69tpoOjVdRQ3K0vVbYyZPCMkq9JuHJSlwY4kIsOqzUiibpRxIZBa-0Qcoyx18RQKVosuec5YSPRZHrMCm8FMaN0G8UK56TwSvtYhhReHGMGxeCvA_xBqcSkxEvzozhaBwrvKzwSq2255FElD72d6m3KT-NX6w7n7RwffWdBkHjrNwkDPwjTVsjo_ede20xZcEDPQI6lVhj5ZqO5cHd58fcFLQX3A1mZ4KC8whM2tdeGUxEYhv1Dttdr_9iKsj2PNuS6MvSlkS1LLFOg0QFb-fZq2L9anr-lGzUY0KNGd8c5eReX3OtkKYw1zaeKmDENt-s4T32TL5nwm07L1EnS2Ijr3Xv2Onyyeio0SEkZv39-ISSetvRqUsfugE5_wdex3ri&lang=sage&interacts=eJyLjgUAARUAuQ==}{interactive {\sc SageMathCell}} where the theoretical Betti diagrams in low degrees are computed.

It is thus natural to pose the following question:

\begin{QU}[\cite{nazgul}]
If $2d\geq4$, do the Betti numbers of a general hyperkähler fourfold $(X,H)\in\cM_{2d}^{(1)}$ coincide with the ones given by the theoretical Betti diagram?
\end{QU}

Note that our \autoref{intro-thmdef}.\eqref{intro-thmdef-1} settles this question in the affirmative for $2d=4$. On the other hand, the same question will not work verbatim in higher divisibility: for a general element of $(X,H)\in\cM_6^{(2)}$, the variety of lines of a cubic fourfold, the 15 quadrics containing $X$---namely the 15 Plücker quadrics containing $\Gr(2,6)$---have unexpected linear relations.

Finally, it would also be interesting to understand very ampleness in higher dimensions. Following \cite{debarre}, let us denote by ${}^{m}\!{\cM}_{2d}^{(\gamma)}$ the moduli space of hyperkähler manifolds of $K3^{[m]}$-type equipped with a polarization of square $2d$ and divisibility $\gamma$. By analogy with the cases $m=1$ and $m=2$, it is natural to expect:

\begin{QU}
    Let $m\geq 3$ and $2d\geq4$. For a general $(X,H)\in{}^{m}\!{\cM}_{2d}^{(1)}$, is the line bundle $H$ very ample? If so, does $H$ define a projectively normal embedding?
\end{QU}

To the best of our knowledge, very ampleness is only known in the range $d\geq m+1$ (see \cite[Corollary 3.9]{debarre}). Note that the analogous question may have a negative answer for higher divisibility: as shown in \cite{epwcubes}, for a general $(X,H)\in{}^{3}\!{\cM}_{4}^{(2)}$ the corresponding morphism $X\to\bP\left(H^0(X,H)^\vee\right)=\bP^{19}$ has degree 2 onto its image.

\bibliography{refer}
\bibliographystyle{alphaspecial}

\end{document}